\documentclass[table]{amsart}
\usepackage[margin=1.1in]{geometry}
\usepackage{amscd,amsmath,amsxtra,amsthm,amssymb,xr,mathrsfs,mathtools,enumerate,orcidlink}
\usepackage{tikz}
\usetikzlibrary{arrows.meta,positioning}

\usetikzlibrary{calc}
\usepackage{tikz-3dplot}
\usepackage{stmaryrd}
\usepackage{multirow}
\usepackage{xcolor}
\usepackage{commath}
\usepackage{comment}
\usepackage{svg}
\usepackage{graphics}
\usepackage{longtable} 
\usepackage{pdflscape} 
\usepackage{booktabs}
\usepackage{hyperref}
\definecolor{vegasgold}{rgb}{0.77, 0.7, 0.35}
\definecolor{darkgoldenrod}{rgb}{0.72, 0.53, 0.04}
\definecolor{gold(metallic)}{rgb}{0.83, 0.69, 0.22}
\hypersetup{
 colorlinks=true,
 linkcolor=darkgoldenrod,
 filecolor=brown,      
 urlcolor=gold(metallic),
 citecolor=darkgoldenrod,
 }
\newtheorem{lthm}{Theorem}

\usepackage[all,cmtip]{xy}

\DeclareFontFamily{U}{wncy}{}
\DeclareFontShape{U}{wncy}{m}{n}{<->wncyr10}{}
\DeclareSymbolFont{mcy}{U}{wncy}{m}{n}
\DeclareMathSymbol{\Sh}{\mathord}{mcy}{"58}
\usepackage[T2A,OT2,T1]{fontenc}

\newtheorem{theorem}{Theorem}[section]
\newtheorem{lemma}[theorem]{Lemma}

\newtheorem*{theorem*}{Theorem}
\newtheorem*{ass*}{Assumption}
\newtheorem{definition}[theorem]{Definition}

\newtheorem{proposition}[theorem]{Proposition}

\newcommand{\Z}{\mathbb{Z}}
\newcommand{\Q}{\mathbb{Q}}
\newcommand{\F}{\mathbb{F}}

\newcommand{\cC}{\mathcal{C}}
\newcommand{\cO}{\mathcal{O}}

\newcommand{\Hom}{\mathrm{Hom}}

\newcommand{\Gr}{\mathrm{Gr}}

\newcommand{\R}{\mathbb{R}}
\providecommand{\C}{}
\renewcommand{\C}{\mathbb{C}}

\newcommand{\cM}{\mathcal{M}}

\newcommand{\Gal}{\operatorname{Gal}}
\newcommand{\Tr}{\operatorname{Tr}}
\newcommand{\disc}{\operatorname{disc}}
\newcommand{\sh}{\operatorname{sh}}
\newcommand{\eps}{\varepsilon}

\newcommand{\Aut}{\operatorname{Aut}}

\title[Shapes of tame abelian number fields]{Shapes of tame abelian number fields}

\author[A.~Ray]{Anwesh Ray\, \orcidlink{0000-0001-6946-1559}}
\address{Chennai Mathematical Institute, Chennai, India}
\email{anwesh@cmi.ac.in}


\begin{document}

\begin{abstract}
The shape of a number field $K$ of degree $n\geq2$ is the similarity
class of the lattice obtained by projecting $\cO_K$ orthogonally
away from the rational line in Minkowski space.
It defines a point in the space of shapes of rank-$(n-1)$ lattices,
$\mathscr S_{n-1}:=\mathrm{GL}_{n-1}(\Z)\backslash
\mathrm{GL}_{n-1}(\R)/\mathrm{GO}_{n-1}(\R)$,
where $\mathrm{GO}_{n-1}(\R)$ is the group of orthogonal similitudes. Let $G$ be a fixed nontrivial finite abelian group, and let $K/\Q$ vary over tamely ramified Galois extensions with Galois group isomorphic to $G$. We study the variation of the shapes of such number fields in $\mathscr S_{n-1}$, where $n=|G|$.
\end{abstract}

\maketitle

\section{Introduction}

The Minkowski embedding realizes the ring of integers of a degree-$n$ number field $K$ as a lattice of rank $n$ in a Euclidean space. Its covolume is determined by the absolute discriminant. To define the \emph{shape} of $K$, one projects this lattice orthogonally away from the line generated by $1$ and takes the resulting lattice up to orthogonal transformations and positive scalar dilations. This records the geometry of the projected lattice up to similarity. For a prescribed arithmetic family, one can then ask which shapes occur and how they are distributed when the fields are ordered by discriminant.

The study of this invariant began with Terr's work on cubic orders \cite{Terr1997}. Bhargava and H subsequently proved equidistribution results for shapes in large families of cubic, quartic and quintic number fields \cite{BhargavaHarron2016}. In Galois families, the symmetries of the field impose additional restrictions on the lattice. Cyclic cubic fields, for example, have a fixed hexagonal shape. H and Harron determined the shapes of Galois quartic fields and also described their distribution in the biquadratic case \cite{HarronHarron2020}. Their results exhibit families whose shapes vary in spaces of much smaller dimension than the ambient space of lattice shapes. In earlier work \cite{JakharRayMultiquadratic2026}, we studied totally real multiquadratic fields in which $2$ is unramified and found that positive diagonal parameters govern the variation of their shapes. 

Let $G$ be a fixed nontrivial finite abelian group and consider the family of isomorphism classes of tamely ramified $G$-extensions $K$ over $\Q$. In this setting, the Hilbert--Speiser theorem supplies a normal integral basis. Two features of tame abelian extensions are especially useful for the shape problem. The first concerns the integral lattice. If $N$ is the conductor of $K$ and $M=\Q(\zeta_N)$, then the Gaussian period $\alpha_K:=\Tr_{M/K}(\zeta_N)$ is a normal integral basis generator. Its conjugates therefore identify $\cO_K$ with the integral group ring $\Z[G]$, through the map $a\mapsto a\alpha_K$. Moreover, $\Tr_{K/\Q}(\alpha_K)=\mu(N)$ is equal to $1$ or $-1$.

The second feature concerns the metric. The regular representation of an abelian group is diagonalized by its characters, so the Minkowski form pulled back to $\R[G]$ can be studied by Fourier analysis. The Fourier coefficient attached to a character $\chi$ is the resolvent
\[
        R_\chi(\alpha_K)
        =\sum_{g\in G}\overline{\chi(g)}g(\alpha_K).
\]
Since $N$ is squarefree, its absolute square is precisely the primitive conductor $f_\chi(K)$ of the character.

Let $\widehat G=\Hom(G,\C^\times)$ and let $\Omega(G)$ be the set of $\Gal(\overline\Q/\Q)$-orbits in $\widehat G\setminus\{1\}$. For $\omega\in\Omega(G)$, let $V_\omega$ denote the corresponding irreducible rational summand of $\Q[G]$. All characters in $\omega$ have the same conductor, which we denote by $f_\omega(K)$. Let $\eps:\Z[G]\to\Z$ be the augmentation map, given by $\eps(\sum_ga_gg)=\sum_ga_g$, and write $N_G=\sum_{g\in G}g$. We shall use the fixed lattice
\begin{equation}\label{eq:LG}
       L_G:=\{\,|G|a-\eps(a)N_G:a\in\Z[G] \,\}
       \subset \ker(\eps)\otimes_\Z\Q.
\end{equation}
For every $\omega$, let $Q_\omega$ be the restriction to $V_\omega\otimes_\Q\R$ of the standard Euclidean form on $\R[G]$ for which the elements of $G$ form an orthonormal basis.

\begin{lthm}\label{thm:intro-main}
Let $G$ be a nontrivial finite abelian group and let $K/\Q$ be a tamely ramified Galois extension with $\Gal(K/\Q)\simeq G$. Then the shape of $K$ is the similarity class of the fixed lattice $L_G$ equipped with the metric
\[
        \bigoplus_{\omega\in\Omega(G)} f_\omega(K)Q_\omega.
\]
In particular, the shape of $K$ is determined by the conductors of the nontrivial rational characters of $G$.
\end{lthm}

Theorem~\ref{thm:intro-main} is proved in a slightly more precise form as Theorem~\ref{thm:main}. Simultaneously multiplying all the $f_\omega$ by the same positive number does not change the shape, so the natural parameter space is the positive torus $T_G:=(\R_{>0})^{\Omega(G)}/\R_{>0}$. Its logarithmic coordinates form $E_G:=\R^{\Omega(G)}/\R\mathbf1$. Theorem~\ref{thm:intro-main} defines a continuous map $\Phi_G:E_G\to\mathscr S_{|G|-1}$ by assigning to $[y]$ the shape of $L_G$ with metric $\bigoplus_\omega e^{y_\omega}Q_\omega$. The tame shape locus lies in the image of this map. Let $\cC(G)$ denote the set of nontrivial cyclic subgroups of $G$. After choosing an identification of the Galois group with $G$, define, for $C\in\cC(G)$,
\begin{equation}\label{eq:qC}
       q_C(K):=\prod_{\substack{p\text{ ramified in }K\\ I_p=C}}p.
\end{equation}
The integers $q_C(K)$ are squarefree and pairwise coprime. A tame character is ramified at $p$ exactly when it is nontrivial on $I_p$, and its conductor exponent is then one. Consequently, if $\chi\in\omega$,
\begin{equation}\label{eq:conductor-inertia-intro}
       f_\omega(K)=
       \prod_{\substack{C\in\cC(G)\\ \chi|_C\neq 1}}q_C(K).
\end{equation}
Taking logarithms turns this multiplicative relation into a linear one. Define $A_G=(a_{\omega,C})_{\omega\in\Omega(G),\,C\in\cC(G)}$ where
\begin{equation}\label{eq:AG-intro}
       a_{\omega,C}=
       \begin{cases}
       1,&\chi|_C\neq1,\\
       0,&\chi|_C=1,
       \end{cases}
\end{equation}
where $\chi$ is any member of the orbit $\omega$. The matrix $A_G$ is shown to be square and nonsingular. The detailed argument appears in Section~\ref{sec:inertia-coordinates}.

Let $\ell$ be the smallest prime divisor of $|G|$. Put $\cM(G)=\{C\leq G:|C|=\ell\}$ and $s_G=|\cM(G)|=(|G[\ell]|-1)/(\ell-1)$. Let $\pi:\R^{\Omega(G)}\to E_G$ be the quotient map. For the column $A_C$ of $A_G$, write $v_C=\pi(A_C)$. The vectors $v_C$ for $C\in\cM(G)$ are affinely independent, and hence determine a simplex
\[
       \mathscr P_G=\operatorname{conv}\{v_C:C\in\cM(G)\}\subset E_G.
\]
For a marking $\iota:\Gal(K/\Q)\xrightarrow{\sim}G$, the labeled logarithmic parameter is $\lambda(K,\iota)=\pi((\log f_\omega(K,\iota))_\omega)$. Changing the marking acts through $\Aut(G)$, so its class in $E_G/\Aut(G)$ is intrinsic to the field. The map $\Phi_G$ is invariant under this action and therefore descends to that finite quotient. The quotient removes the choice of marking; its identification with the full lattice-shape locus is not needed for the statements below.

\begin{lthm}\label{thm:intro-distribution}
Let $\ell$ be the smallest prime divisor of $|G|$, and put $a_0=|G|(1-1/\ell)$ and $\alpha_G=1/a_0$. Let $\mathscr F_G^{\rm mark}(X)$ be the marked tame $G$-extensions of absolute discriminant at most $X$. Then
\[
       |\mathscr F_G^{\rm mark}(X)|\sim C_G^{\rm mark}X^{\alpha_G}(\log X)^{s_G-1}
\]
for an explicit constant $C_G^{\rm mark}>0$. Moreover, the rescaled logarithmic parameters $a_0\lambda(K,\iota)/\log X$ become uniformly distributed on the simplex $\mathscr P_G$.
\end{lthm}

Taken together, these results give a description of how the arithmetic of a tame abelian extension is reflected in the geometry of its shape. The underlying lattice depends only on the Galois group, while the variation in its metric is governed by the conductors of the nontrivial rational characters. The distribution theorem shows, moreover, that the logarithmic conductor parameters are equi-distributed with respect to the uniform measure on a simplex. It would be interesting to understand how this picture changes when the tameness assumption is removed. Wild ramification introduces additional difficulties, particularly because normal integral bases need no longer exist, and it is not clear whether an equally explicit description of the shape locus can be obtained. A further direction is to investigate nonabelian Galois extensions, where the decomposition into rational representations may provide a starting point for understanding the geometry and distribution of shapes in more general Galois families.

\subsection*{Declarations}
\begin{description}
     \item[Conflict of interest] Not applicable; there is no conflict of interest to report.
    \item[Availability of data and materials] No data was generated or analyzed in obtaining the results in this article.
    \item[Funding] No funding was received during the preparation of this article.
    \item[Use of artificial intelligence]
The author used ChatGPT 6 Astra to assist in correcting mistakes in the article and to assist in improving and editing exposition while the article was being written. It was not used to generate the mathematical ideas or the main results of the article.
\end{description}

\section{Preliminary notions}\label{sec:shape-prelim}

We recall the definition of the shape space and the character decomposition of an abelian group.

Given a full rank lattice in an inner product space, its shape is invariant with respect to rotations, reflections and positive scalar dilations. The shape is then viewed as a point in the space of shapes of lattices of fixed rank. This space is a double quotient space and comes equipped with a natural measure. We follow the conventions used in
\cite{DasKalaMukhopadhyayRay2026,JakharKalwaniyaRayRoy2026,
JakharRayOctic2026,JakharRayMultiquadratic2026}. This will also fix
the notation that we use throughout the paper.

Let $V$ be an $r$-dimensional real inner product space and let
$\Lambda, \Lambda'\subset V$ be lattices of full rank. We say that $\Lambda$ and $\Lambda'$ have the same shape if $\Lambda$ can be
obtained from $\Lambda'$ by a composition of an orthogonal
transformation and a positive scalar dilation. After choosing an orthonormal basis of $V$, we may identify $V$ with
$\R^r$. A choice of basis for $\Lambda$ then identifies the lattice
with an element of $\mathrm{GL}_r(\R)$. Changing the $\Z$-basis of the lattice
gives the action of $\mathrm{GL}_r(\Z)$, while rotations, reflections and
positive scalar dilations give the action of the group of orthogonal
similitudes
\[
        \mathrm{GO}_r(\R)
        :=
        \left\{
        g\in\mathrm{GL}_r(\R):
        g^{\mathrm t}g=cI_r
        \text{ for some }c>0
        \right\}.
\]
The space of shapes of rank $r$ lattices is the double
quotient
\begin{equation}\label{eq:space-of-shapes}
        \mathscr S_r
        :=
        \mathrm{GL}_r(\Z)\backslash
        \mathrm{GL}_r(\R)/
        \mathrm{GO}_r(\R).
\end{equation}
There is an equivalent description in terms of positive-definite
quadratic forms which is often more useful for computations. Let $\mathscr P_r$ be the space of positive-definite symmetric $r\times r$ real
matrices. If $\{b_1,\ldots,b_r\}$
is a basis of $\Lambda$, its Gram matrix is
\[
        \Gr(\Lambda)
        :=
        \bigl(\langle b_i,b_j\rangle\bigr)_{1\leq i,j\leq r}
        \in\mathscr P_r.
\]
Changing the integral basis by a matrix
$U\in\mathrm{GL}_r(\Z)$ replaces the Gram matrix $G=\Gr(\Lambda)$ by $U^{\mathrm t}GU$ whereas multiplying the lattice by a common positive scalar replaces
$G$ by a positive scalar multiple of $G$. Consequently, the same
space of shapes may be described as
\begin{equation}\label{eq:shape-space-gram}
        \mathscr S_r
        \simeq
        \mathrm{GL}_r(\Z)\backslash
        \mathscr P_r/\R_{>0}.
\end{equation}

It is sometimes convenient to remove the scalar ambiguity explicitly.
Set
\[
        \mathscr H_r
        :=
        \left\{
        G\in\mathscr P_r:\det(G)=1
        \right\}.
\]
Every positive-definite symmetric matrix has a unique positive scalar
multiple lying in $\mathscr H_r$. Hence
\[
        \mathscr P_r/\R_{>0}
        \simeq
        \mathscr H_r,
\]
and we may equally well regard
\[
        \mathscr S_r
        \simeq
        \mathrm{GL}_r(\Z)\backslash\mathscr H_r.
\]

The Haar measure on $\mathrm{GL}_r(\R)$ gives rise to a natural invariant
measure on $\mathscr S_r$. We denote it by
$\mu_{\mathscr S_r}$, or simply by $\mu$ when the ambient space of
shapes is clear from the context. In the families considered later in
the paper, the shapes will usually lie on lower-dimensional subsets of
$\mathscr S_r$, and the relevant limiting measures will be obtained
from the natural measures on these loci.

The description of the space of shapes becomes especially concrete in
rank $2$. A general element of $\mathscr H_2$ may be written in the form $Y=
        \begin{pmatrix}
        a&b\\
        b&c
        \end{pmatrix}$, where $a>0$, $c>0$ and $ac-b^2=1$. There is a natural identification of $\mathscr H_2$ with the complex
upper half-plane
\[
        \mathbb H
        =
        \{z=x+iy\in\mathbb C:y>0\}.
\]
Indeed, to the matrix $Y$ above we associate the point $z=x+iy$, where $x=\frac{b}{c}$ and $y=\sqrt{\frac{a}{c}-x^2}$. Since $ac-b^2=1$, the second expression may also be written simply as $y=\frac{1}{c}$. Conversely, if $z=x+iy\in\mathbb H$, then the corresponding
determinant-one positive-definite form is
\begin{equation}\label{eq:H2-upper-half-plane}
        Y_z
        =
        \frac{1}{y}
        \begin{pmatrix}
        x^2+y^2&x\\
        x&1
        \end{pmatrix}.
\end{equation}

Under this correspondence, the $\operatorname{GL}_2(\mathbb Z)$-orbits of
positive-definite forms correspond to modular orbits in the upper half-plane.
For the convention $Y\mapsto U^{\mathrm t}YU$ used here, if
$U=\left(\begin{smallmatrix}a&b\\c&d\end{smallmatrix}\right)$ has determinant $1$,
the corresponding point is $(az+c)/(bz+d)$. If $\det U=-1$, it is
$(a\overline z+c)/(b\overline z+d)$. Consequently, the
rank-$2$ space of shapes may be identified with the modular surface
\begin{equation}\label{eq:S2-modular-surface}
        \mathscr S_2
        =
        \operatorname{GL}_2(\mathbb Z)\backslash\mathscr H_2
        \simeq
        \operatorname{GL}_2(\mathbb Z)\backslash\mathbb H.
\end{equation}
One should note here the use of $\operatorname{GL}_2(\mathbb Z)$
rather than $\operatorname{SL}_2(\mathbb Z)$. The determinant $-1$
elements introduce a reflection which identifies the two halves of
the standard $\operatorname{SL}_2(\mathbb Z)$-fundamental domain. We
may therefore take
\begin{equation}\label{eq:modular-fundamental-domain}
        \mathcal F_2
        =
        \left\{
        x+iy\in\mathbb H:
        0\leq x\leq\frac12,\quad
        x^2+y^2\geq1
        \right\}
\end{equation}
as a closed fundamental region for the space of rank-$2$ shapes. The two vertical sides are reflection boundaries of the quotient; in particular, the side $x=1/2$ must be included to represent its boundary orbits. The measure induced from the Haar measure is, up to normalization, the
hyperbolic measure $d\mu(z)
        =
        \frac{dx\,dy}{y^2}$.

        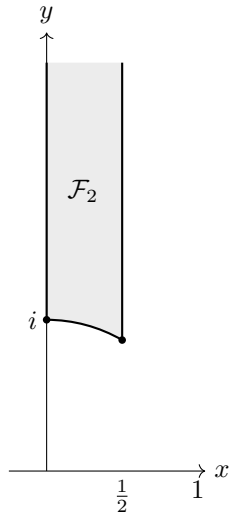
\begin{figure}[ht]
\centering
\begin{tikzpicture}[scale=2.0]

\fill[gray!15]
  (0,1)
  arc[start angle=90,end angle=60,radius=1]
  -- (0.5,2.7)
  -- (0,2.7)
  -- cycle;

\draw[->] (-0.25,0) -- (1.05,0)
  node[right] {$x$};
\draw[->] (0,0) -- (0,2.9)
  node[above] {$y$};

\draw[thick]
  (0,1)
  arc[start angle=90,end angle=60,radius=1];

\draw[thick] (0,1) -- (0,2.7);
\draw[thick] (0.5,{sqrt(3)/2}) -- (0.5,2.7);

\fill (0,1) circle (0.025);
\fill (0.5,{sqrt(3)/2}) circle (0.025);

\node[left] at (0,1) {$i$};
\node[below] at (0.5,0) {$\frac12$};
\node[below] at (1,0) {$1$};

\node[right] at (0.5,{sqrt(3)/2})
  {};

\node at (0.24,1.85) {$\mathcal F_2$};

\node[rotate=-28] at (0.73,0.72)
  {};

\end{tikzpicture}
\caption{A fundamental domain $\mathcal F_2$ for
$\operatorname{GL}_2(\mathbb Z)\backslash\mathbb H$.}
\label{fig:fundamental-domain-F2}
\end{figure}
\par We now specialize this discussion to number fields. Let $K$ be a
number field of degree $n=[K:\Q]$ and let $\cO_K$ denote its ring of integers. The Minkowski embedding realizes $\cO_K$ as a lattice of rank $n$ in
the $n$-dimensional real vector space $K_\R:=K\otimes_\Q\R$. We equip $K_\R$ with the natural positive-definite inner product
\begin{equation}\label{eq:minkowski-form}
        \langle x,y\rangle_K
        :=
        \sum_{\sigma:K\hookrightarrow\C}
        \sigma(x)\overline{\sigma(y)}.
\end{equation}
Although the individual terms on the right-hand side of
\eqref{eq:minkowski-form} may be complex, the sum is real for
$x,y\in K$, since the nonreal embeddings occur in complex-conjugate
pairs. When $K$ is totally real, it simplifies to
\[
        \langle x,y\rangle_K
        =
        \Tr_{K/\Q}(xy).
\]

The line $\R\cdot1$ is distinguished by the inclusion $\Q\subset K$. We project away from this line when defining the shape. For $x\in K$, we have
\[
        \langle x,1\rangle_K
        =
        \Tr_{K/\Q}(x).
\]
It follows that the orthogonal complement of the line $\R\cdot 1$ is
the trace-zero hyperplane
\[
        K_\R^0
        :=
        \left\{
        x\in K_\R:
        \Tr_{K/\Q}(x)=0
        \right\}.
\]
The orthogonal projection of $x\in K_\R$ onto this hyperplane is
\[
        x\longmapsto
        x-\frac{\Tr_{K/\Q}(x)}{n}.
\]
Rather than introduce denominators, we clear the factor $n$ and write
\begin{equation}\label{eq:trace-zero-projection}
        x^\perp
        :=
        nx-\Tr_{K/\Q}(x).
\end{equation}
We then define the projected lattice
\begin{equation}\label{eq:projected-lattice}
        \cO_K^\perp
        :=
        \left\{
        x^\perp:x\in\cO_K
        \right\}.
\end{equation}

Observe that the kernel of the map $\cO_K\longrightarrow K_\R^0$ defined by $x\mapsto x^\perp$
is precisely $\Z$. Indeed, if $x^\perp=0$, then $nx=\Tr_{K/\Q}(x)\in\Q$, and hence $x\in\Q$. Since $x$ is integral, it follows that
$x\in\Z$. Conversely, every rational integer lies in the kernel.
Thus $\cO_K^\perp$ is a lattice of rank $n-1$ in the
$(n-1)$-dimensional Euclidean space $K_\R^0$.

\begin{definition}\label{def:shape-number-field}
Let $K$ be a number field of degree $n\geq2$. The \emph{shape} of $K$,
denoted by $\sh(K)$, is the similarity class of the projected Minkowski lattice
$\cO_K^\perp$. Thus $\sh(K)\in\mathscr S_{n-1}$.
\end{definition}

More explicitly, let $\{1,\alpha_1,\ldots,\alpha_{n-1}\}$ be an integral basis of $\cO_K$. The elements $\alpha_1^\perp,\ldots,\alpha_{n-1}^\perp$ then give a $\Z$-basis of $\cO_K^\perp$. The shape of $K$ is 
represented by the Gram matrix
\begin{equation}\label{eq:shape-gram-number-field}
        \Gr(K)
        :=
        \left(
        \left\langle
        \alpha_i^\perp,\alpha_j^\perp
        \right\rangle_K
        \right)_{1\leq i,j\leq n-1}.
\end{equation}

{\leavevmode\color{black}From now on, let $G$ be a nontrivial finite abelian group of order $n$, and let $\widehat G=\Hom(G,\C^\times)$ be its character group. Every irreducible complex representation of $G$ is one-dimensional.} Given a commutative ring $R$, let $R[G]$ be the group ring of $G$ over $R$. For $\chi\in \widehat G$, we define
\begin{equation}\label{eq:e-chi}
       e_\chi=\frac1n\sum_{g\in G}\overline{\chi(g)}g\in\C[G].
\end{equation}
The elements $e_\chi$ are the primitive orthogonal idempotents of $\C[G]$. They satisfy $e_\chi e_\psi=0$ for $\chi\neq\psi$ and $\sum_\chi e_\chi=1$. Consequently
\[
       \C[G]=\bigoplus_{\chi\in\widehat G}\C e_\chi.
\]

The decomposition over $\Q$ is obtained by grouping characters into Galois orbits. If $\omega\subset\widehat G$ is an orbit under $\Gal(\overline\Q/\Q)$, set
\[
       e_\omega=\sum_{\chi\in\omega}e_\chi.
\]
The coefficients of $e_\omega$ are fixed by the absolute Galois group, so $e_\omega\in\Q[G]$. We define
\[
       V_\omega=e_\omega\Q[G].
\]Let $\chi\in\widehat G$ have order $d$.  Its values lie in
$\Q(\zeta_d)$, and the action
\[
   \Gal(\Q(\zeta_d)/\Q)\simeq(\Z/d\Z)^\times
\]
on $\widehat G$ is given by $\chi\mapsto\chi^a$ where $a\in(\Z/d\Z)^\times$. Consequently, the Galois orbit of $\chi$ consists precisely of the
generators of the cyclic subgroup $\langle\chi\rangle\leq\widehat G$.
Thus, if $\omega$ is such an orbit and the characters in $\omega$
have order $d_\omega$, then $|\omega|=\varphi(d_\omega)$. The corresponding irreducible rational representation $V_\omega$
therefore has
\[
        \dim_\Q V_\omega
        =
        |\omega|
        =
        \varphi(d_\omega).
\]

The trivial character gives the line $\Q N_G$, where $N_G=\sum_{g\in G}g$. We shall use repeatedly the rational decomposition
\begin{equation}\label{eq:rational-decomposition}
        \Q[G]=\Q N_G\oplus\bigoplus_{\omega\in\Omega(G)}V_\omega.
\end{equation}
Let $B_0$ denote the standard Euclidean form on $\R[G]$ for which the elements of $G$ are orthonormal. Since left multiplication by $G$ permutes this basis, $B_0$ is $G$-invariant.

\begin{lemma}\label{lem:character-lines-orthogonal}
Let $B$ be a $G$-invariant Hermitian form on $\C[G]$. Then the complex character lines $\C e_\chi$ and $\C e_\psi$ are orthogonal whenever $\chi\neq\psi$.
\end{lemma}

\begin{proof}
Consider nonzero vectors $v_\chi\in \C e_\chi$ and $v_\psi\in \C e_\psi$. For every $g\in G$, we have that $gv_\chi=\chi(g)v_\chi$ and $gv_\psi=\psi(g)v_\psi$.
We find that
\begin{align*}
        B(v_\chi,v_\psi)
        &=B(gv_\chi,gv_\psi)\\
        &=B(\chi(g)v_\chi,\psi(g)v_\psi)\\
        &=\chi(g)\overline{\psi(g)}B(v_\chi,v_\psi).
\end{align*}
Since $\chi\neq\psi$, there is a $g\in G$ for which $\chi(g)\overline{\psi(g)}\neq1$. The displayed identity then gives $B(v_\chi,v_\psi)=0$.
\end{proof}

\section{Computing the shape from the integral basis}
\label{sec:normal-basis}

In this section, we provide an explicit description of the shapes of tamely ramified abelian number fields with prescribed Galois group $G$.  The Hilbert--Speiser theorem is used to describe the integral basis of such a number field. For a positive integer $N$, put $\zeta_N=e^{2\pi i/N}$. Let $K/\Q$ be a finite abelian extension with Galois group $G$, and let
$N$ be its conductor, namely, the smallest positive integer such that $K\subseteq\Q(\zeta_N)$. We set $M=\Q(\zeta_N)$, $\Gamma=\Gal(M/\Q)$, $H=\Gal(M/K)$ and thus $G\simeq\Gamma/H$.
\begin{proposition}\label{prop:hilbert-speiser}
With this notation, the following assertions hold:
\begin{enumerate}
    \item $K/\Q$ is tamely ramified if and only if $N$ is odd and squarefree.
    \item If the equivalent conditions in (1) hold, the Gaussian period
\begin{equation}\label{eq:alphaK}
        \alpha_K
        :=
        \Tr_{M/K}(\zeta_N)
        =
        \sum_{h\in H}h(\zeta_N)
\end{equation}
is a normal integral basis generator for $K/\Q$.
\item Under the same assumption,
\begin{equation}\label{eq:trace-alpha}
        \Tr_{K/\Q}(\alpha_K)=\mu(N)\in\{\pm1\}.
\end{equation}
\end{enumerate}

\end{proposition}

\begin{proof}
Assertions (1) and (2) are well known, see
\cite[Theorem~8.4]{JohnstonGaloisModules}. We find that
\begin{align*}
        \Tr_{K/\Q}(\alpha_K)
        &=
        \Tr_{K/\Q}\left(\Tr_{M/K}(\zeta_N)\right)\\
        &=
        \Tr_{M/\Q}(\zeta_N)=\mu(N);
\end{align*}
cf. \cite[\S8.3]{ApostolANT}.  Under the conditions in (1), $N$ is squarefree and $\mu(N)=\pm1$.
This proves \eqref{eq:trace-alpha}.
\end{proof}

For the rest of this section, assume that $K/\Q$ is tamely ramified.
The normal integral basis identifies $\cO_K$ with $\Z[G]$. We first use this
identification to express the projected lattice as a fixed lattice $L_G$.
We then compute the pulled-back Minkowski form on each character line;
its coefficient will be the conductor of that character.
\par We set $N_G:=\sum_{g\in G}g\in\Z[G]$ and let $\eps:\Z[G]\longrightarrow\Z$ denote the augmentation map defined by
       \[ \eps\left(\sum_{g\in G}a_gg\right)
        =
        \sum_{g\in G}a_g.
\] We let
\[
        L_G
        :=
        \left\{
        na-\eps(a)N_G:a\in\Z[G]
        \right\}
        \subseteq\ker(\eps)\otimes_\Z\Q,
\]
where $n=|G|=[K:\Q]$ and let \[\iota_K:\Z[G]\xrightarrow{\sim}\cO_K\] be defined by $a\mapsto a\alpha_K$. It follows from Proposition \ref{prop:hilbert-speiser} that $\iota_K$ is an isomorphism. For \[a=\sum_{g\in G}a_gg\in\Z[G],\] one has that
\begin{align}
        \Tr_{K/\Q}\bigl(\iota_K(a)\bigr)
        &=
        \sum_{g\in G}a_g
        \Tr_{K/\Q}\bigl(g(\alpha_K)\bigr)\notag\\
        &=
        \mu(N)\sum_{g\in G}a_g
        =
        \mu(N)\eps(a).
        \label{eq:trace-iota}
\end{align}
On the other hand,
\begin{equation}\label{eq:NG-image}
        \iota_K(N_G)
        =
        \sum_{g\in G}g(\alpha_K)
        =
        \Tr_{K/\Q}(\alpha_K)
        =
        \mu(N).
\end{equation}
Here, as usual, the rational integer $\mu(N)$ is viewed as the corresponding
element of $\cO_K$.

\par We can now compare the scaled orthogonal projection of $\cO_K$ with $L_G$.
For $a\in\Z[G]$, equations \eqref{eq:trace-iota} and
\eqref{eq:NG-image} give
\begin{align}
        n\iota_K(a)
        -
        \Tr_{K/\Q}\bigl(\iota_K(a)\bigr)
        &=
        n\iota_K(a)-\mu(N)\eps(a)\notag\\
        &=
        \iota_K(na)-\eps(a)\iota_K(N_G)\notag\\
        &=
        \iota_K\bigl(na-\eps(a)N_G\bigr).
        \label{eq:projected-group-ring}
\end{align}
This therefore shows that
\begin{equation}\label{eq:projected-identification}
        \cO_K^\perp=\iota_K(L_G).
\end{equation}

\par Extend $\iota_K$ to $\iota_K:\C[G]\xrightarrow{\sim}K\otimes_\Q\C$, where $G$ acts on the $K$-factor and acts trivially
on the scalar factor $\C$. For a character
\[
        \chi\in\widehat G=\Hom(G,\C^\times),
\]
use the inclusion $K\subseteq\Q(\zeta_N)\subset\C$ to define the complex resolvent
\begin{equation}\label{eq:resolvent}
        R_\chi(\alpha_K)
        :=
        \sum_{g\in G}
        \overline{\chi(g)}\,g(\alpha_K).
\end{equation}
The character $\chi$ may be inflated through
\[
        \Gamma\longrightarrow\Gamma/H\simeq G.
\]
Using the canonical identification
\[
        \Gamma\simeq(\Z/N\Z)^\times,
\]
we may therefore regard $\chi$ as a Dirichlet character modulo $N$.

Let $f_\chi=f_\chi(K)$ denote the conductor of the primitive Dirichlet
character which induces this character modulo $N$.  Since $f_\chi\mid N$,
and $N$ is squarefree, the quotient $N/f_\chi$ is automatically coprime to
$f_\chi$.

We denote by
\[
        \tau_q(\eta)
        :=
        \sum_{a\bmod q}
        \eta(a)e^{2\pi ia/q}
\]
the Gauss sum of a Dirichlet character $\eta$ modulo $q$, where
$\eta(a)=0$ whenever $(a,q)>1$.

Expanding the definition of $\alpha_K$ and grouping the elements of
$\Gamma$ by cosets of $H$, we obtain
\begin{align}
        R_\chi(\alpha_K)
        &=
        \sum_{g\in G}
        \overline{\chi(g)}
        \sum_{h\in H}gh(\zeta_N)\notag\\
        &=
        \sum_{\gamma\in\Gamma}
        \overline{\chi(\gamma)}
        \gamma(\zeta_N)\notag\\
        &=
        \sum_{a\in(\Z/N\Z)^\times}
        \overline{\chi(a)}\zeta_N^a
        =
        \tau_N(\overline\chi).
        \label{eq:resolvent-gauss}
\end{align}
Thus the size of the resolvent is governed by an imprimitive Gauss sum.

For reference, the standard theory of induced characters and Gauss sums is
developed in \cite[\S\S8.5--8.10]{ApostolANT}.

\begin{lemma}\label{lem:imprimitive-gauss}
Let $\chi$ be a Dirichlet character modulo a squarefree integer $N$, and
suppose that $\chi$ is induced by a primitive character $\chi^*$ of conductor
$f$.  Writing $N=fm$, note that $(f,m)=1$. Then we find that
\begin{equation}\label{eq:imprimitive-gauss}
        \tau_N(\chi)
        =
        \mu(m)\chi^*(m)\tau_f(\chi^*).
\end{equation}
In particular,
\begin{equation}\label{eq:imprimitive-gauss-abs}
        |\tau_N(\chi)|^2=f.
\end{equation}
\end{lemma}

\begin{proof}
Since $N$ is squarefree and $f\mid N$, the integers $f$ and
$m=N/f$ are coprime.  Choose integers $\overline m$ and $\overline f$
satisfying
\[
        m\overline m\equiv1\pmod f,
        \qquad
        f\overline f\equiv1\pmod m.
\]
By the Chinese remainder theorem, a residue class $a$ modulo $N=fm$ is
determined by the pair
\[
        (u,v)=(a\bmod f,a\bmod m),
\]
and one may choose the representative
\[
        a\equiv
        um\overline m+vf\overline f
        \pmod{fm}.
\]
Consequently,
\begin{align*}
        e^{2\pi ia/N}
        &=
        e^{2\pi i u\overline m/f}
        e^{2\pi i v\overline f/m}\\
        &=
        \zeta_f^{u\overline m}
        \zeta_m^{v\overline f}.
\end{align*}

If $a$ is a unit modulo $N$, then both $u$ and $v$ are units modulo
$f$ and $m$, respectively.  Since $\chi$ is induced from $\chi^*$, we have
\[
        \chi(a)=\chi^*(u).
\]
It follows that
\begin{align}
        \tau_N(\chi)
        &=
        \sum_{a\in(\Z/N\Z)^\times}
        \chi(a)\zeta_N^a\notag\\
        &=
        \left(
        \sum_{u\in(\Z/f\Z)^\times}
        \chi^*(u)\zeta_f^{u\overline m}
        \right)
        \left(
        \sum_{v\in(\Z/m\Z)^\times}
        \zeta_m^{v\overline f}
        \right).
        \label{eq:gauss-factorization}
\end{align}

We evaluate the two factors separately.  Since $\overline m$ is a unit
modulo $f$, multiplication by $\overline m$ permutes the reduced residue
classes modulo $f$.  The standard twisting identity for Gauss sums therefore
gives
\begin{align*}
        \sum_{u\in(\Z/f\Z)^\times}
        \chi^*(u)\zeta_f^{u\overline m}
        &=
        \overline{\chi^*}(\overline m)\tau_f(\chi^*).
\end{align*}
Since
\[
        m\overline m\equiv1\pmod f,
\]
we have
\[
        \chi^*(\overline m)=\chi^*(m)^{-1}.
\]
Character values have absolute value one, and hence
\[
        \overline{\chi^*}(\overline m)
        =
        \chi^*(m).
\]

For the second factor in \eqref{eq:gauss-factorization}, multiplication by
the unit $\overline f$ permutes $(\Z/m\Z)^\times$.  Therefore
\begin{align*}
        \sum_{v\in(\Z/m\Z)^\times}
        \zeta_m^{v\overline f}
        &=
        \sum_{v\in(\Z/m\Z)^\times}\zeta_m^v\\
        &=
        c_m(1)
        =
        \mu(m),
\end{align*}
where the last identity is the Ramanujan-sum evaluation recalled above.
Combining the two factors gives
\[
        \tau_N(\chi)
        =
        \mu(m)\chi^*(m)\tau_f(\chi^*),
\]
which proves \eqref{eq:imprimitive-gauss}.

For completeness, we also recall why a primitive Gauss sum has absolute
value $\sqrt f$.  Let $\eta$ be a primitive Dirichlet character modulo $f$
and, for $n\in\Z$, put
\[
        S(n)
        :=
        \sum_{a\bmod f}
        \eta(a)e^{2\pi ian/f}.
\]
If $(n,f)=1$, then multiplication by $n$ permutes the reduced residue
classes modulo $f$.  Making the substitution $a\mapsto n^{-1}a$ gives
\[
        S(n)
        =
        \overline{\eta(n)}\tau_f(\eta).
\]
If $(n,f)>1$, then $S(n)=0$.  Indeed, put $d=(n,f)$.  Since $\eta$ is
primitive, it cannot be trivial on the kernel of the reduction map from
$(\Z/f\Z)^\times$ to $(\Z/(f/d)\Z)^\times$.  Hence there exists a unit
$b\bmod f$ satisfying
\[
        b\equiv1\pmod{f/d},
        \qquad
        \eta(b)\neq1.
\]
The first congruence implies
\[
        bn\equiv n\pmod f.
\]
Replacing $a$ by $ba$ in the definition of $S(n)$ therefore gives
\[
        S(n)=\overline{\eta(b)}S(n),
\]
and since $\eta(b)\neq1$, it follows that $S(n)=0$.

Now apply Parseval's identity to the finite Fourier transform of the function
$a\mapsto\eta(a)$ on $\Z/f\Z$.  We obtain
\[
        \sum_{n\bmod f}|S(n)|^2
        =
        f\sum_{a\bmod f}|\eta(a)|^2
        =
        f\varphi(f).
\]
There are exactly $\varphi(f)$ residue classes $n\bmod f$ with
$(n,f)=1$, and on each of them
\[
        |S(n)|=|\tau_f(\eta)|.
\]
Consequently
\[
        \varphi(f)|\tau_f(\eta)|^2
        =
        f\varphi(f),
\]
and therefore
\begin{equation}\label{eq:primitive-gauss-absolute}
        |\tau_f(\eta)|^2=f.
\end{equation}
This is the classical absolute-value formula for primitive Gauss sums; see
also \cite[Theorem~8.15]{ApostolANT}.

Applying \eqref{eq:primitive-gauss-absolute} to $\eta=\chi^*$, and using
\[
        |\mu(m)|=|\chi^*(m)|=1,
\]
gives \eqref{eq:imprimitive-gauss-abs}.
\end{proof}

Applying Lemma~\ref{lem:imprimitive-gauss} to $\overline\chi$ in
\eqref{eq:resolvent-gauss}, whose primitive conductor is again $f_\chi(K)$,
we obtain the fundamental identity
\begin{equation}\label{eq:gauss-absolute}
        |R_\chi(\alpha_K)|^2=f_\chi(K).
\end{equation}

Let $B_0$ denote the standard positive-definite Hermitian form on
$\C[G]$ for which the elements of $G$ form an orthonormal basis.  Thus
\[
        B_0(g,h)=
        \begin{cases}
        1,&g=h,\\
        0,&g\neq h.
        \end{cases}
\]

The normal-basis isomorphism
\[
        \iota_K:\Z[G]\xrightarrow{\sim}\cO_K
\]
extends uniquely by $\R$-linearity to an isomorphism
\[
        \iota_{K,\R}:
        \R[G]\xrightarrow{\sim}K\otimes_\Q\R.
\]
We use this map to transport the Minkowski inner product from
$K\otimes_\Q\R$ to $\R[G]$.  Thus, for $a,b\in\R[G]$, we define
\begin{equation}\label{eq:BK-definition}
        B_K(a,b)
        :=
        \bigl\langle\iota_{K,\R}(a),\iota_{K,\R}(b)\bigr\rangle_K.
\end{equation}
In this way $\iota_{K,\R}$ becomes an isometry between
$(\R[G],B_K)$ and $(K\otimes_\Q\R,\langle\,\cdot,\cdot\,\rangle_K)$.

In order to use the character decomposition of the group algebra, we
next pass from $\R[G]$ to its complexification
\[
        \C[G]
        =
        \R[G]\otimes_\R\C.
\]
The real inner product $B_K$ has a unique Hermitian extension to
$\C[G]$.  We continue to denote this extension by $B_K$.  Explicitly,
if $a_1,\ldots,a_s,b_1,\ldots,b_t\in\R[G]$ and
$z_1,\ldots,z_s,w_1,\ldots,w_t\in\C$, then
\[
        B_K\left(
        \sum_{i=1}^s z_i a_i,
        \sum_{j=1}^t w_j b_j
        \right)
        =
        \sum_{i=1}^s\sum_{j=1}^t
        z_i\overline{w_j}\,B_K(a_i,b_j).
\]
Thus $B_K$ is linear in the first variable and conjugate-linear in the
second.

Because $K/\Q$ is Galois, the action of $G$ permutes the embeddings of $K$.
Hence $B_K$ is $G$-invariant:
\begin{equation}\label{eq:BK-invariant}
        B_K(\sigma x,\sigma y)=B_K(x,y)
        \qquad
        (\sigma\in G).
\end{equation}
The same is clearly true of $B_0$.

For $\chi\in\widehat G$, put
\begin{equation}\label{eq:vchi}
        v_\chi
        :=
        \sum_{g\in G}\overline{\chi(g)}g.
\end{equation}
Then for $\sigma\in G$ we have that
\[
        \sigma v_\chi
        =
        \chi(\sigma)v_\chi,
\]
and therefore $v_\chi$ spans the one-dimensional $\chi$-eigenspace in $\C[G]$. If
$\chi\neq\psi$, then there exists $\sigma\in G$ such that
$\chi(\sigma)\neq\psi(\sigma)$.  By $G$-invariance,
\begin{align*}
        B_K(v_\chi,v_\psi)
        &=
        B_K(\sigma v_\chi,\sigma v_\psi)\\
        &=
        \chi(\sigma)\overline{\psi(\sigma)}
        B_K(v_\chi,v_\psi).
\end{align*}
Since
\[
        \chi(\sigma)\overline{\psi(\sigma)}\neq1,
\]
we must have
\[
        B_K(v_\chi,v_\psi)=0.
\]
Thus the character lines are mutually orthogonal.  This is also the
one-dimensional abelian case of Schur's lemma; see
\cite[Chapter~2, \S2.2]{Serre1977}.  The identical argument applies to
$B_0$.

We next compute the length of $v_\chi$.  Since the elements of $G$ form an
orthonormal basis for $B_0$,
\begin{align}
        B_0(v_\chi,v_\chi)
        &=
        \sum_{g\in G}|\chi(g)|^2
        =
        n.
        \label{eq:B0-vchi}
\end{align}
On the other hand,
\hypersetup{linkcolor=black,citecolor=black,urlcolor=black}\[
        \iota_K(v_\chi)
        =\sum_{g\in G}g(\alpha_K)\otimes\overline{\chi(g)}
        \in K\otimes_\Q\C.
\]
For $\sigma\in G$, let $\widetilde\sigma:K\otimes_\Q\C\to\C$ be the
$\C$-linear map $x\otimes z\mapsto\sigma(x)z$, using the fixed inclusion
$K\subset\C$. These maps give the embedding coordinates on the complexification,
so the Hermitian extension of the Minkowski form is the sum of the products
of corresponding coordinates, conjugating the second coordinate. We have
\begin{align*}
        \widetilde\sigma\bigl(\iota_K(v_\chi)\bigr)
        &=\sum_{g\in G}\overline{\chi(g)}\sigma g(\alpha_K)\\
        &=\chi(\sigma)\sum_{h\in G}\overline{\chi(h)}h(\alpha_K)\\
        &=\chi(\sigma)R_\chi(\alpha_K).
\end{align*}
Thus all embedding coordinates have absolute value $|R_\chi(\alpha_K)|$.
Using \eqref{eq:gauss-absolute}, we obtain
\begin{align}
        B_K(v_\chi,v_\chi)
        &=
        \sum_{\sigma\in G}
        \left|
        \widetilde\sigma\bigl(\iota_K(v_\chi)\bigr)
        \right|^2\notag\\
        &=
        \sum_{\sigma\in G}
        |\chi(\sigma)|^2
        |R_\chi(\alpha_K)|^2\notag\\
        &=
        n|R_\chi(\alpha_K)|^2\notag\\
        &=
        nf_\chi(K).\\
        \label{eq:BK-vchi}
\end{align}
Comparing \eqref{eq:B0-vchi} and \eqref{eq:BK-vchi}, and using the fact
that the $\chi$-eigenspace is one-dimensional, gives
\begin{equation}\label{eq:metric-character-line}
        B_K|_{\C v_\chi}
        =
        f_\chi(K)B_0|_{\C v_\chi}.
\end{equation}

Next we group the complex character lines into rational blocks.
Suppose that $\chi'$ is Galois conjugate to $\chi$.  If $\chi$ has order
$d$, then $\chi'=\chi^a$ for some $a\in(\Z/d\Z)^\times$. The characters $\chi$ and $\chi^a$ have the same kernel on every inertia
subgroup. Hence $f_{\chi'}(K)=f_\chi(K)$ and thus the conductor is constant on every Galois orbit
$\omega\subseteq\widehat G$, and we denote the common value by $f_\omega(K)$.

Recall that $V_\omega$ is the rational irreducible summand corresponding to
the orbit $\omega$, and that $Q_\omega$ is the restriction of $B_0$ to
$V_\omega\otimes_\Q\R$.  Grouping
\eqref{eq:metric-character-line} over all $\chi\in\omega$ gives
\begin{equation}\label{eq:metric-rational-block}
        B_K|_{V_\omega\otimes_\Q\R}
        =
        f_\omega(K)Q_\omega.
\end{equation}
Denote by $\Omega(G)$ the set of Galois orbits of nontrivial complex characters of $G$. We have therefore proved the following key result.

\begin{theorem}\label{thm:main}
Let $G$ be a nontrivial finite abelian group of order $n$, and let $K/\Q$ be a
tamely ramified $G$-extension of conductor $N$.  Put
\[
        \alpha_K
        =
        \Tr_{\Q(\zeta_N)/K}(\zeta_N).
\]
Then the following statements hold.
\begin{enumerate}
\item
The conjugates of $\alpha_K$ form a normal integral basis of $\cO_K$.

\item
Under the resulting identification
\[
        \iota_K:\Z[G]\xrightarrow{\sim}\cO_K,
\]
the scaled projected lattice $\cO_K^\perp$ is identified with the fixed
lattice $L_G$.

\item
With respect to the rational decomposition
\[
        \ker(\eps)\otimes_\Z\R
        =
        \bigoplus_{\omega\in\Omega(G)}
        V_\omega\otimes_\Q\R,
\]
the pulled-back Minkowski form is
\begin{equation}\label{eq:BK-block-decomposition}
        B_K|_{\ker(\eps)\otimes\R}
        =
        \bigoplus_{\omega\in\Omega(G)}
        f_\omega(K)Q_\omega.
\end{equation}
Consequently,
\begin{equation}\label{eq:shape-main}
        \sh(K)
        =
        \sh\left(
        L_G,
        \bigoplus_{\omega\in\Omega(G)}
        f_\omega(K)Q_\omega
        \right).
\end{equation}
This is to say that when $\cO_K^\perp$ is identified with $L_G$ the shape is given by the form
\[ \bigoplus_{\omega\in\Omega(G)}
        f_\omega(K)Q_\omega.\]
\end{enumerate}
\end{theorem}

\begin{proof}
The first assertion is Proposition~\ref{prop:hilbert-speiser}. Choose a marking $\iota:\Gal(K/\Q)\xrightarrow{\sim}G$.
Together with the normal integral basis generator $\alpha_K$, this gives the
$\Z[G]$-module isomorphism $\iota_K:\Z[G]\xrightarrow{\sim}\cO_K$ defined by $a\mapsto a\alpha_K$. By Proposition~\ref{prop:hilbert-speiser},
\[
        \Tr_{K/\Q}(\alpha_K)=\mu(N),
\]
and the calculation in \eqref{eq:trace-iota}--\eqref{eq:projected-group-ring}
shows that
\[
        n\iota_K(a)
        -
        \Tr_{K/\Q}\bigl(\iota_K(a)\bigr)
        =
        \iota_K\bigl(na-\eps(a)N_G\bigr)
\]
and hence $\cO_K^\perp=\iota_K(L_G)$ which proves the second assertion.

Equations \eqref{eq:B0-vchi} and \eqref{eq:BK-vchi} give $B_0(v_\chi,v_\chi)=n$ and $B_K(v_\chi,v_\chi)=nf_\chi(K)$. The character lines are pairwise orthogonal for both $B_0$ and $B_K$. The conductor $f_\chi(K)$ is constant on a Galois orbit
$\omega\subseteq\widehat G$. Therefore we obtain
\[
        B_K|_{\ker(\eps)\otimes_\Z\R}
        =
        \bigoplus_{\omega\in\Omega(G)}
        f_\omega(K)Q_\omega
\]
and this proves the third assertion. The shape of $K$ is the similarity class of $\cO_K^\perp$ equipped with the metric $B_K$.  Identifying $\cO_K^\perp$ with $L_G$ gives
\eqref{eq:shape-main}.
\end{proof}

The lattice $L_G$ and the forms
$Q_\omega$ are fixed once $G$ has been fixed.  The field $K$ gives rise to the tuple $\bigl(f_\omega(K)\bigr)_{\omega\in\Omega(G)}$.
Multiplying all of these scalars by the same positive number multiplies the
entire metric by that number and therefore leaves its similarity class
unchanged.  Thus the natural parameter space for these metrics is the
positive diagonal torus
\begin{equation}\label{eq:TG}
        T_G
        :=
        (\R_{>0})^{\Omega(G)}/\R_{>0},
\end{equation}
where $\R_{>0}$ acts diagonally.

It is convenient to pass to logarithmic coordinates.  Put $E_G
        :=
        \R^{\Omega(G)}/\R\mathbf1$ where $\mathbf1$ is the vector whose coordinates are all equal to $1$.
For a marked $G$-extension $(K,\iota)$, define its logarithmic conductor
parameter by
\begin{equation}\label{eq:lambdaK}
        \lambda(K,\iota)
        :=
        \left[
        \bigl(\log f_\omega(K)\bigr)_{\omega\in\Omega(G)}
        \right]
        \in E_G.
\end{equation}

The corresponding map $\Phi_G:E_G\longrightarrow\mathscr S_{n-1}$ is given by
\begin{equation}\label{eq:shape-parameter-map}
        [y]\longmapsto
        \sh\left(
        L_G,
        \bigoplus_{\omega\in\Omega(G)}
        e^{y_\omega}Q_\omega
        \right).
\end{equation}
This is well defined, since replacing every $y_\omega$ by
$y_\omega+c$ multiplies the entire metric by the common factor $e^c$.
The map is continuous, and Theorem~\ref{thm:main} says precisely
that
\begin{equation}\label{eq:shape-via-Phi}
        \sh(K)=\Phi_G\bigl(\lambda(K,\iota)\bigr).
\end{equation}

There is a finite symmetry which appears when the marking is forgotten.
An automorphism $\theta\in\Aut(G)$
permutes the standard basis of $\Z[G]$.  It therefore preserves the
augmentation map, the element $N_G$, and the lattice $L_G$.  It also
permutes the complex characters, hence the rational character blocks
$V_\omega$, and preserves the standard form $B_0$.  Thus $\theta$ induces
an integral isometry between the metrics obtained by applying the
corresponding permutation to the coordinates
$(y_\omega)_{\omega\in\Omega(G)}$.  Consequently $\Phi_G$ is $\Aut(G)$-invariant and descends to a continuous map
\[
        E_G/\Aut(G)\longrightarrow\mathscr S_{n-1}.
\]
In particular, forgetting the marking is compatible with forgetting the
labels on the conductor coordinates.  We do not claim that every accidental
similarity between two lattices in the image of $\Phi_G$ must arise from an
automorphism of $G$.

Consider the normal integral basis
\[
        \{g(\alpha_K):g\in G\}
\]
of $\cO_K$.  With respect to the corresponding basis of $\Z[G]$, the
standard form $B_0$ has determinant $1$.  After complexifying, the form
$B_K$ is diagonal on the character lines, and
\[
        B_K|_{\C v_\chi}
        =
        f_\chi(K)B_0|_{\C v_\chi}.
\]
The trivial character has conductor $1$.  Taking determinants therefore
gives
\[
        \det(B_K)
        =
        \prod_{\chi\in\widehat G}f_\chi(K)
        =
        \prod_{\chi\in\widehat G\setminus\{1\}}f_\chi(K).
\]
On the other hand, the determinant of the Minkowski Gram matrix of an
integral basis is the absolute discriminant of $K$.  Hence
\begin{equation}\label{eq:conductor-discriminant}
        |\disc(K)|
        =
        \prod_{\chi\in\widehat G\setminus\{1\}}
        f_\chi(K).
\end{equation}
Grouping the characters into Galois orbits gives the equivalent expression
\begin{equation}\label{eq:conductor-discriminant-orbits}
        |\disc(K)|
        =
        \prod_{\omega\in\Omega(G)}
        f_\omega(K)^{|\omega|}
        =
        \prod_{\omega\in\Omega(G)}
        f_\omega(K)^{\dim_\Q V_\omega}.
\end{equation}

\section{Cyclic inertia coordinates}\label{sec:inertia-coordinates}

The character conductors diagonalize the metric, but for counting fields it is more natural to organize the ramified primes by inertia subgroup. We now pass between these two coordinate systems.

Let $p$ be a finite prime ramified in $K$, and let $I_p\leq G$ be its inertia subgroup. Since $K/\Q$ is tame, $I_p$ is cyclic and its order is prime to $p$. A character $\chi\in\widehat G$ is unramified at $p$ precisely when its restriction to $I_p$ is trivial. If the restriction is nontrivial, the character is tamely ramified and its local conductor exponent is $1$. Thus the global conductor is obtained simply by multiplying those ramified primes on which the character is nontrivial. Let $\cC(G)$ denote the set of nontrivial cyclic subgroups of $G$. Recall from \eqref{eq:qC} that $q_C(K)$ is the product of the
ramified primes whose inertia subgroup is $C$, with the empty
product understood to be $1$.

\begin{proposition}\label{prop:conductor-inertia}
For $\omega\in\Omega(G)$ and $\chi\in\omega$,
\begin{equation}\label{eq:conductor-inertia}
       f_\omega(K)=
       \prod_{\substack{C\in\cC(G)\\\chi|_C\neq1}}q_C(K).
\end{equation}
Equivalently, if $x_C=\log q_C(K)$ and $y_\omega=\log f_\omega(K)$, then $y=A_Gx$.
\end{proposition}

\begin{proof}
Fix a finite prime $p$.  Since $K/\Q$ is tame, the wild inertia group at $p$ is trivial.  For a one-dimensional character $\chi$ of $G$, the local Artin conductor exponent is therefore
\[
        a_p(\chi)=\dim(\C)-\dim(\C^{I_p})
        =\begin{cases}
        0,&\chi|_{I_p}=1,\\
        1,&\chi|_{I_p}\neq1.
        \end{cases}
\]
This can also be read directly in Dirichlet-character language: a tame local character is either unramified, in which case $p$ does not divide its conductor, or nontrivial on tame inertia, in which case its conductor contains $p$ to the first power and no higher power.

It follows that the primitive conductor of $\chi$ is
\begin{equation}\label{eq:fchi-local-product}
       f_\chi=
       \prod_{\substack{p\text{ ramified in }K\\\chi|_{I_p}\neq1}}p.
\end{equation}
Now fix a nontrivial cyclic subgroup $C\leq G$.  By definition,
\[
        q_C(K)=\prod_{I_p=C}p.
\]
The primes in \eqref{eq:fchi-local-product} are partitioned by their inertia subgroups.  Grouping them according to this partition gives
\[
       f_\chi=
       \prod_{\substack{C\in\cC(G)\\\chi|_C\neq1}}q_C(K),
\]
which is \eqref{eq:conductor-inertia}.

The condition $\chi|_C\neq1$ is invariant under Galois conjugation of $\chi$.  Indeed, if $\chi'$ is a Galois conjugate of $\chi$, then $\chi'=\chi^a$ for some integer $a$ relatively prime to the order of $\chi$.  If $\chi(c)$ has order $e$, then $(a,e)=1$, so $\chi(c)^a=1$ if and only if $\chi(c)=1$.  Hence the product depends only on the orbit $\omega$, as claimed.  Finally, taking logarithms of the product gives
\[
        \log f_\omega(K)=
        \sum_{C\in\cC(G)}a_{\omega,C}\log q_C(K),
\]
which is exactly the matrix relation $y=A_Gx$.
\end{proof}
A Galois orbit of a nontrivial character consists of all generators of one cyclic subgroup of $\widehat G$. Hence $\Omega(G)$ may be identified with the set of nontrivial cyclic subgroups of $\widehat G$. A finite abelian group is noncanonically isomorphic to its character group, so $G$ and $\widehat G$ have the same number of cyclic subgroups of each order. In particular
\[
       |\Omega(G)|=|\cC(G)|.
\]
Thus $A_G$ is square.

\begin{theorem}\label{thm:incidence}
Let $G$ be a nontrivial finite abelian group, and put
$m_G=|\cC(G)|=|\Omega(G)|$. Then
\[
        A_G\in\mathrm{GL}_{m_G}(\Q).
\]
In particular, for every tamely ramified $G$-extension $K/\Q$
with a fixed marking $\iota:\Gal(K/\Q)\xrightarrow{\sim}G$,
the vectors
\[
        x(K,\iota)
        :=\bigl(\log q_C(K,\iota)\bigr)_{C\in\cC(G)}\quad \text{and}\quad
        y(K,\iota)
        :=\bigl(\log f_\omega(K,\iota)\bigr)_{\omega\in\Omega(G)}
\]
satisfy $y(K,\iota)=A_Gx(K,\iota)$. Moreover,
\[
        \sh(K)
        =\Phi_G\bigl([A_Gx(K,\iota)]\bigr).
\]
Consequently,
\[
        \bigl\{
          \sh(K):
          K/\Q \text{ is tamely ramified and }
          \Gal(K/\Q)\simeq G
        \bigr\}
        \subseteq \Phi_G(E_G).
\]
\end{theorem}
\begin{proof}
We first prove that $A_G$ is nonsingular. Include the trivial
character and the trivial subgroup by setting
\[
        \Omega^+(G)=\Omega(G)\cup\{1\}\quad \text{and}\quad
        \cC^+(G)=\cC(G)\cup\{1\}.
\]
Here the additional element of $\Omega^+(G)$ denotes the orbit
consisting of the trivial character. Define
\[
        B^+=(b_{\omega,C})_{
          \omega\in\Omega^+(G),\,C\in\cC^+(G)},
      \quad \text{where}\quad
        b_{\omega,C}
        =
        \begin{cases}
        1,&\chi|_C=1,\\
        0,&\chi|_C\neq1,
        \end{cases}
\]
where $\chi$ is any member of $\omega$.

It is easy to see that this definition does not depend on the choice of
$\chi$. Indeed, if $\chi$ has order $d$, then every character in its Galois
orbit has the form $\chi^a$ with $(a,d)=1$. For every $g\in G$,
the value $\chi(g)$ is a $d$th root of unity, and therefore
\[
        \chi(g)^a=1
        \quad\Longleftrightarrow\quad
        \chi(g)=1.
\]
Thus $\chi$ and $\chi^a$ have the same kernel. In particular,
their restrictions to $C$ are either both trivial or both
nontrivial.

We shall also use the equality
\[
        |\Omega^+(G)|=|\cC^+(G)|=m_G+1.
\]
Indeed, the Galois orbit of a character $\chi$ consists precisely
of the generators of the cyclic subgroup $\langle\chi\rangle$
of $\widehat G$. Consequently, $\Omega^+(G)$ is in bijection
with the cyclic subgroups of $\widehat G$, including the trivial
subgroup. To compare their number with that for $G$, write
\[
        G\simeq C_{d_1}\times\cdots\times C_{d_r}.
\]
The character group of $C_{d_i}$ is cyclic of order $d_i$, so
$\widehat G$ has the same invariant factors as $G$. Hence $G$
and $\widehat G$ are isomorphic, although no particular
isomorphism is needed here. They therefore have the same number
of cyclic subgroups. In particular, $B^+$ is square.

For $\omega\in\Omega^+(G)$, let $V_\omega$ be the corresponding
irreducible rational representation, and write
\[
        \theta_\omega=\sum_{\chi\in\omega}\chi
\]
for its character. Since $G$ is abelian, its complex irreducible
representations are one-dimensional, and
\[
        V_\omega\otimes_\Q\C
        \simeq
        \bigoplus_{\chi\in\omega}\C e_\chi,
        \qquad
        g e_\chi=\chi(g)e_\chi.
\]
In particular,
\[
        \dim_\Q V_\omega=|\omega|.
\]
One can see this decomposition explicitly by choosing
$\chi\in\omega$ of order $d$: the representation $V_\omega$
is realized on $\Q(\zeta_d)$, with $g\in G$ acting by
multiplication by $\chi(g)$. The embeddings of $\Q(\zeta_d)$
into $\C$ give the character lines indexed by $\omega$.

Taking $C$-invariants commutes with extension of scalars from
$\Q$ to $\C$. Indeed, the averaging operator
\[
        \frac{1}{|C|}\sum_{c\in C}c
\]
is a projection onto the invariant subspace over either field.
If $b_{\omega,C}=1$, then every character in $\omega$ is trivial
on $C$, and every summand $\C e_\chi$ is fixed by $C$.
If $b_{\omega,C}=0$, none of these character lines has a
nonzero $C$-fixed vector. Consequently,
\begin{equation}\label{eq:fixed-dim-incidence}
        \dim_\Q(V_\omega^C)
        =(\dim_\Q V_\omega)b_{\omega,C}.
\end{equation}

For each $C\in\cC^+(G)$, let
\[
        \rho_C:=\operatorname{Ind}_C^G(1)
\]
be the character of the rational permutation representation
$\Q[G/C]$. Frobenius reciprocity gives, for every
$\chi\in\widehat G$,
\[
        \langle\rho_C,\chi\rangle_G
        =
        \langle 1,\chi|_C\rangle_C
        =
        \frac{1}{|C|}
        \sum_{c\in C}\overline{\chi(c)}
        =
        \begin{cases}
        1,&\chi|_C=1,\\
        0,&\chi|_C\neq1.
        \end{cases}
\]
The last equality follows from the character sum over $C$:
the sum is $|C|$ for the trivial character and zero for a
nontrivial character. Thus
\[
        \rho_C
        =
        \sum_{\substack{\chi\in\widehat G\\\chi|_C=1}}\chi
        =
        \sum_{\omega\in\Omega^+(G)}
        b_{\omega,C}\theta_\omega.
\]
It follows that the column of $B^+$ indexed by $C$ is the
coefficient vector of $\rho_C$ in the basis of irreducible
rational characters.

For completeness, these coefficients are compatible with
\eqref{eq:fixed-dim-incidence} as follows. Orthogonality of
the complex irreducible characters gives
\[
        \langle\theta_\omega,\theta_{\omega'}\rangle_G
        =
        \begin{cases}
        |\omega|,&\omega=\omega',\\
        0,&\omega\neq\omega'.
        \end{cases}
\]
Therefore
\[
        \langle\theta_\omega,\rho_C\rangle_G
        =
        |\omega|b_{\omega,C}
        =
        \dim_\Q(V_\omega^C).
\]
Thus the coefficient of $\theta_\omega$ in $\rho_C$ is
$b_{\omega,C}$, while the dimension of the fixed subspace is
$|\omega|b_{\omega,C}$.

We now prove directly that the characters $\rho_C$, with
$C\in\cC^+(G)$, are linearly independent. Enumerate the cyclic
subgroups as
\[
        C_1,\ldots,C_{m_G+1},
        \qquad
        |C_1|\leq\cdots\leq|C_{m_G+1}|,
\]
and choose a generator $g_i$ of each $C_i$. For the trivial
subgroup, take $g_i=1$. Since $G$ is abelian, an element
$g\in G$ fixes a coset $hC_j$ if and only if
\[
        ghC_j=hC_j
        \quad\Longleftrightarrow\quad
        g\in C_j.
\]
Thus $g$ fixes every coset when $g\in C_j$, and fixes no coset
otherwise. The permutation character consequently satisfies
\[
        \rho_{C_j}(g_i)
        =
        \begin{cases}
        [G:C_j],&C_i\subseteq C_j,\\
        0,&C_i\not\subseteq C_j.
        \end{cases}
\]

Consider the square evaluation matrix
\[
        M=\bigl(\rho_{C_j}(g_i)\bigr)_
          {1\leq i,j\leq m_G+1}.
\]
If $i>j$, then $|C_i|\geq|C_j|$. Inclusion
$C_i\subseteq C_j$ would force equality of the two subgroups
when their orders are equal, and is impossible when
$|C_i|>|C_j|$. Since the enumeration has no repetitions,
$M_{i,j}=0$ for $i>j$. Hence $M$ is upper triangular, with
diagonal entries
\[
        M_{i,i}=[G:C_i].
\]
It follows that
\[
        \det M=\prod_{i=1}^{m_G+1}[G:C_i]\neq0.
\]
Any linear relation among the characters $\rho_{C_j}$,
evaluated at $g_1,\ldots,g_{m_G+1}$, would therefore have all
its coefficients equal to zero.

There are $m_G+1$ such permutation characters and $m_G+1$
irreducible rational characters. Hence the $\rho_C$ form a
basis of the rational character space. This verifies directly
the basis assertion obtained from Artin induction in the
present abelian setting; cf.\ \cite[Chapter~9]{Serre1977}.
Their coefficient matrix $B^+$ is therefore nonsingular
over $\Q$.

Order the trivial character and the trivial subgroup first
in $B^+$. The trivial character is trivial on every subgroup,
and every character is trivial on the trivial subgroup.
Consequently,
\[
        B^+
        =
        \begin{pmatrix}
        1&\mathbf1^{\mathrm t}\\
        \mathbf1&B
        \end{pmatrix},
\]
where $B$ is indexed by $\Omega(G)$ and $\cC(G)$, and
$\mathbf1$ is the column vector of $m_G$ ones. By the
definition of $A_G$,
\[
        A_G=\mathbf1\mathbf1^{\mathrm t}-B.
\]
Subtracting the first row of $B^+$ from each of its remaining
rows gives
\[
        \begin{pmatrix}
        1&\mathbf1^{\mathrm t}\\
        0&B-\mathbf1\mathbf1^{\mathrm t}
        \end{pmatrix}
        =
        \begin{pmatrix}
        1&\mathbf1^{\mathrm t}\\
        0&-A_G
        \end{pmatrix}.
\]
These row operations preserve the determinant, so
\[
        \det(B^+)=\det(-A_G)
        =(-1)^{m_G}\det(A_G).
\]
Since $\det(B^+)\neq0$, we obtain $\det(A_G)\neq0$.
The entries of $A_G$ are integers, and hence its inverse has
rational entries. This proves
\[
        A_G\in\mathrm{GL}_{m_G}(\Q).
\]

Now fix a marked tame $G$-extension $(K,\iota)$. By
Proposition~\ref{prop:conductor-inertia}, for every
$\omega\in\Omega(G)$,
\[
        f_\omega(K,\iota)
        =
        \prod_{C\in\cC(G)}
        q_C(K,\iota)^{a_{\omega,C}}.
\]
Each $q_C(K,\iota)$ is positive, with value $1$ when no
ramified prime has inertia subgroup $C$. Taking logarithms
therefore gives
\[
        \log f_\omega(K,\iota)
        =
        \sum_{C\in\cC(G)}
        a_{\omega,C}\log q_C(K,\iota).
\]
These identities, one for each $\omega$, are exactly
\[
        y(K,\iota)=A_Gx(K,\iota).
\]
In particular, nonsingularity also gives
\[
        x(K,\iota)=A_G^{-1}y(K,\iota).
\]

Finally, recall that
\[
        E_G=\R^{\Omega(G)}/\R\mathbf1
\]
and that $\Phi_G$ is defined by
\[
        \Phi_G([y])
        =
        \sh\left(
        L_G,
        \bigoplus_{\omega\in\Omega(G)}
        e^{y_\omega}Q_\omega
        \right).
\]
This expression depends only on $[y]$: replacing $y$ by
$y+c\mathbf1$ multiplies the entire quadratic form by $e^c$,
which leaves its similarity class unchanged. For
$y=y(K,\iota)$, we have
$e^{y_\omega}=f_\omega(K,\iota)$. Theorem~\ref{thm:main}
therefore gives
\[
        \sh(K)
        =
        \sh\left(
        L_G,
        \bigoplus_{\omega\in\Omega(G)}
        f_\omega(K,\iota)Q_\omega
        \right)
        =
        \Phi_G([y(K,\iota)])
        =
        \Phi_G([A_Gx(K,\iota)]).
\]
Every tame extension whose Galois group is isomorphic to
$G$ admits a marking, and the displayed equality places its
shape in $\Phi_G(E_G)$. This proves the final inclusion.
\end{proof}

\section{Distribution of the shapes}\label{sec:distribution}

We now order the tame $G$-extensions by discriminant and study the resulting distribution in logarithmic conductor-parameter space. Throughout this section $G$ is a fixed nontrivial finite abelian group of order $n$.

For $C\in\cC(G)$ write $d_C=|C|$ and define
\begin{equation}\label{eq:aC}
        a_C:=n\left(1-\frac1{d_C}\right).
\end{equation}
These are exactly the exponents with which primes of inertia type $C$ enter the discriminant.

\begin{lemma}\label{lem:tame-discriminant}
Let $p$ be ramified in a tame $G$-extension $K/\Q$, and suppose $I_p=C$. Then
\[
        v_p(\disc(K))=a_C.
\]
Consequently
\begin{equation}\label{eq:disc-inertia}
        |\disc(K)|=
        \prod_{C\in\cC(G)}q_C(K)^{a_C}.
\end{equation}
\end{lemma}

\begin{proof}
Fix a prime $\mathfrak p$ of $K$ above $p$.  Since $K/\Q$ is Galois, the ramification index $e=e(\mathfrak p/p)$, the residue degree $f=f(\mathfrak p/p)$, and the number $g$ of primes above $p$ are independent of the choice of $\mathfrak p$, and
\[
        n=[K:\Q]=efg.
\]
Moreover the inertia subgroup $I_p$ has order $e$, so $e=d_C$.

Let $K_{\mathfrak p}/\Q_p$ be the corresponding local extension.  For a finite Galois extension of local fields, the exponent of the different is
\[
        d(\mathfrak p/p)=\sum_{i\geq0}(|G_i|-1),
\]
where $G_i$ are the lower ramification groups.  Tameness means $G_1=1$, while $G_0=I_p$ has order $e$.  Hence
\[
        d(\mathfrak p/p)=e-1.
\]
The exponent of $p$ in the global discriminant is obtained by taking the norm of the global different.  Each of the $g$ primes above $p$ contributes residue degree $f$, and therefore
\begin{align*}
        v_p(\disc(K))
        &=\sum_{\mathfrak p\mid p}f(\mathfrak p/p)
          d(\mathfrak p/p)\\
        &=gf(e-1)\\
        &=\frac{n}{e}(e-1)\\
        &=n\left(1-\frac1e\right)
         =a_C.
\end{align*}
This proves the local formula.  Finally, every ramified prime belongs to exactly one inertia class $C$, and
\[
        q_C(K)=\prod_{I_p=C}p.
\]
Multiplying the local contributions over all ramified primes and grouping the primes according to $C$ gives
\[
        |\disc(K)|
        =\prod_{C\in\cC(G)}q_C(K)^{a_C}.
\]
\end{proof}

Let $\ell$ be the smallest prime divisor of $|G|$. Since every nontrivial cyclic subgroup of $G$ has order at least $\ell$, the smallest possible discriminant exponent is
\begin{equation}\label{eq:a0-alpha}
        a_0=n\left(1-\frac1\ell\right)
        =\frac{n(\ell-1)}{\ell}.
\end{equation}
We set $\alpha_G:=\frac1{a_0}
        =\frac{\ell}{n(\ell-1)}$. Equality $a_C=a_0$ holds precisely when $|C|=\ell$. We call these the \emph{minimal inertia subgroups} and put
\[
        \cM(G):=\{C\leq G:|C|=\ell\}.
\]
The group $G[\ell]$ is an $\F_\ell$-vector space, and the members of $\cM(G)$ are exactly its one-dimensional subspaces. Hence
\begin{equation}\label{eq:number-minimal}
        s_G:=|\cM(G)|
        =\frac{|G[\ell]|-1}{\ell-1}.
\end{equation}

We now pass from positive character conductors to logarithmic shape coordinates. Set
\[
        E_G:=\R^{\Omega(G)}/\R\mathbf1,
\]
and let $\pi:\R^{\Omega(G)}\to E_G$ be the quotient map. If $A_C$ denotes the column of $A_G$ indexed by $C$, put
\[
        v_C:=\pi(A_C)\in E_G.
\]
To label the character blocks and inertia subgroups by the fixed group $G$, we choose for the moment an isomorphism $\iota:\Gal(K/\Q)\xrightarrow{\sim}G$. We write $f_\omega(K,\iota)$ and $q_C(K,\iota)$ for the resulting labeled invariants. Proposition~\ref{prop:conductor-inertia} gives the logarithmic shape parameter
\begin{equation}\label{eq:log-shape}
        \lambda(K,\iota)
        :=\pi\bigl((\log f_\omega(K,\iota))_\omega\bigr)
        =\sum_{C\in\cC(G)}(\log q_C(K,\iota))v_C.
\end{equation}
The quotient by $\R\mathbf1$ exactly removes simultaneous scaling of all character blocks, and hence $E_G$ is the natural additive parameter space for the positive conductor torus. An automorphism of $G$ permutes the rational character blocks and the cyclic subgroups, and therefore acts linearly on $E_G$. Replacing $\iota$ by $\alpha\circ\iota$, with $\alpha\in\Aut(G)$, replaces $\lambda(K,\iota)$ by the corresponding $\Aut(G)$-translate. Thus the class
\[
        \overline\lambda(K)\in E_G/\Aut(G)
\]
is canonical for the unmarked field. This finite quotient is the natural parameter space when one counts fields without choosing a marking. The actual lattice shape is its image under the descended map $\Phi_G$; no classification of all integral similarities is being assumed.

Equations \eqref{eq:disc-inertia} and \eqref{eq:log-shape} express the discriminant and the logarithmic conductor parameter in the same coordinates. A factor $Y$ in $q_C$ contributes $Y^{a_C}$ to the discriminant and $(\log Y)v_C$ to the parameter. The exponent $a_C$ is smallest precisely for the order-$\ell$ subgroups. The counting argument below shows that these coordinates determine the limiting simplex.

We next turn the family of fields into an explicit family of squarefree tuples. It is convenient to count marked extensions first.

\begin{definition}
A \emph{marked $G$-extension} is a pair $(K,\iota)$ consisting of a Galois extension $K/\Q$ and an isomorphism $\iota:\Gal(K/\Q)\xrightarrow{\sim}G$.
\end{definition}

For a tuple $\mathbf q=(q_C)_{C\in\cC(G)}$ of positive integers, call $\mathbf q$ \emph{admissible} if the following hold:
\begin{enumerate}
\item every $q_C$ is squarefree;
\item the integers $q_C$ are pairwise coprime;
\item every prime $p\mid q_C$ satisfies $p\equiv1\pmod{d_C}$.
\end{enumerate}
Write $\omega(m)$ for the number of distinct prime divisors of a positive integer $m$, with $\omega(1)=0$. For an admissible tuple define
\begin{equation}\label{eq:weight-q}
        w(\mathbf q):=
        \prod_{C\in\cC(G)}\varphi(d_C)^{\omega(q_C)}.
\end{equation}

\begin{lemma}\label{lem:tame-parametrization}
Marked tame $G$-extensions of $\Q$ are parametrized by admissible tuples $\mathbf q$, together with, for every prime $p\mid q_C$, a surjection $\F_p^\times\twoheadrightarrow C$, subject to the generation condition
\begin{equation}\label{eq:generation-condition}
        \langle C:q_C>1\rangle=G.
\end{equation}
An admissible tuple satisfying \eqref{eq:generation-condition} gives exactly $w(\mathbf q)$ marked extensions. After summing over any collection of marked extensions which is stable under changing the marking, the corresponding number of unmarked fields is obtained by dividing by $|\Aut(G)|$.
\end{lemma}

\begin{proof}
Let $(K,\iota)$ be a marked tame $G$-extension and let $N$ be the conductor of $K$. Proposition~\ref{prop:hilbert-speiser} shows that $N$ is odd and squarefree.  Put $M=\Q(\zeta_N)$.  Restriction of automorphisms gives a surjection
\[
        \Gal(M/\Q)\longrightarrow\Gal(K/\Q),
\]
and after composing with the marking $\iota$ we obtain a surjective homomorphism
\begin{equation}\label{eq:KW-map}
        \rho:(\Z/N\Z)^\times\simeq\Gal(M/\Q)
        \longrightarrow G.
\end{equation}
Conversely, the fixed field of the kernel of any surjective map of the form \eqref{eq:KW-map} is a marked abelian extension with Galois group $G$.  We now describe exactly what tameness and squarefreeness of $N$ do to this map.

Because $N$ is squarefree, the Chinese remainder theorem gives a direct product decomposition
\begin{equation}\label{eq:CRT-local-factors}
        (\Z/N\Z)^\times
        \simeq\prod_{p\mid N}(\Z/p\Z)^\times
        =\prod_{p\mid N}\F_p^\times.
\end{equation}
For a fixed prime $p\mid N$, the factor $\F_p^\times$ in \eqref{eq:CRT-local-factors} is exactly the inertia subgroup at $p$ in the full cyclotomic extension $M/\Q$.  Indeed,
\[
        M=\Q(\zeta_p)\,\Q(\zeta_{N/p}).
\]
The first factor is totally ramified at $p$ with ramification group $(\Z/p\Z)^\times$, while the second factor is unramified at $p$ because its discriminant is supported on the primes dividing $N/p$.  Since the two factors are linearly disjoint, the inertia subgroup of $M/\Q$ at $p$ is precisely the $p$-factor in \eqref{eq:CRT-local-factors}.  Passing to the quotient $G$ shows that the image of this factor under $\rho$ is the inertia subgroup $I_p$ of $K/\Q$.

Let this image be $C$.  The restriction of $\rho$ to the $p$-factor is then a surjection
\[
        \rho_p:\F_p^\times\twoheadrightarrow C.
\]
Since $\F_p^\times$ is cyclic of order $p-1$, a surjection onto the cyclic group $C$ of order $d_C$ exists if and only if
\[
        d_C\mid p-1,
\]
or equivalently $p\equiv1\pmod{d_C}$.  If this divisibility holds, choose a generator $g_p$ of $\F_p^\times$.  A homomorphism is determined by $\rho_p(g_p)$, and it is surjective exactly when this image is a generator of $C$.  There are $\varphi(d_C)$ such generators.  Thus there are exactly $\varphi(d_C)$ local surjections with prescribed image $C$.

Group the ramified primes according to their inertia subgroups. For each $C\in\cC(G)$ set
\[
        q_C=\prod_{\substack{p\mid N\\\rho(\F_p^\times)=C}}p.
\]
Each prime divisor of $N$ appears in one and only one of these products.  Hence every $q_C$ is squarefree and the integers $q_C$ are pairwise coprime.  The preceding paragraph shows that every prime divisor of $q_C$ is congruent to $1$ modulo $d_C$.  Thus the tuple is admissible.

Conversely, suppose an admissible tuple $\mathbf q=(q_C)_C$ is given.  Put
\[
        N=\prod_C q_C.
\]
For every prime $p\mid q_C$, choose one of the $\varphi(d_C)$ surjections
\[
        \rho_p:\F_p^\times\twoheadrightarrow C\hookrightarrow G.
\]
Using the product decomposition \eqref{eq:CRT-local-factors}, the product of these local maps defines a homomorphism
\[
        \rho:(\Z/N\Z)^\times\longrightarrow G.
\]
Its image is generated by the images of the local factors, namely by the cyclic subgroups $C$ for which $q_C>1$.  Consequently $\rho$ is surjective if and only if
\[
        \langle C:q_C>1\rangle=G.
\]
When this condition holds, the fixed field of $\ker(\rho)$ has conductor exactly $N$: every $p\mid N$ has nontrivial local image, so it ramifies, and because the extension is tame its conductor exponent is one.  Thus no prime can be removed from $N$.  The resulting extension is therefore a tame marked $G$-extension with the prescribed inertia tuple.

The choices at distinct primes are independent.  For a fixed admissible tuple satisfying the generation condition, the number of collections of local surjections is therefore
\[
        \prod_C\varphi(d_C)^{\omega(q_C)}=w(\mathbf q),
\]
which proves the asserted weight formula.

Finally, let $K$ be an unmarked field with $\Gal(K/\Q)\simeq G$.  The set of isomorphisms
\[
        \Gal(K/\Q)\xrightarrow{\sim}G
\]
is a torsor under $\Aut(G)$, and hence contains exactly $|\Aut(G)|$ elements.  Changing the marking generally permutes the labels on the cyclic subgroups and therefore changes the labeled tuple.  For this reason division by $|\Aut(G)|$ is legitimate only after one has summed over a collection stable under changing the marking.  All passages from marked to unmarked fields below will be made in this way.
\end{proof}

The local multiplicity $\varphi(d_C)$ is the reciprocal of the Dirichlet density of the primes $p\equiv1\pmod{d_C}$. Their product is therefore one. This explains the simple pole associated with an inertia coordinate; for the minimal coordinates, we establish the required factorization below.

Fix a prime $\ell$, an integer $s\geq1$, and a positive integer $Q$. Consider tuples $(q_1,\ldots,q_s)$ of pairwise coprime squarefree integers such that $(q_1\cdots q_s,Q)=1$ and every prime divisor of every $q_i$ is congruent to $1$ modulo $\ell$. Give such a tuple the weight
\[
        w(\mathbf q)=(\ell-1)^{\omega(q_1)+\cdots+\omega(q_s)}.
\]
The multivariable Dirichlet series is
\begin{equation}\label{eq:minimal-euler-product}
        \mathcal F_Q(\mathbf z)
        =\prod_{\substack{p\nmid Q\\p\equiv1\;({\rm mod}\,\ell)}}
        \left(1+(\ell-1)\sum_{j=1}^sp^{-z_j}\right).
\end{equation}

\begin{lemma}\label{lem:zell-factorization}
There is a neighbourhood $U$ of $1$ and a holomorphic nonvanishing function $h_\ell$ on $U$ such that
\begin{equation}\label{eq:zell-factorization}
        Z_\ell(z):=
        \prod_{p\equiv1\;({\rm mod}\,\ell)}
        (1-p^{-z})^{-(\ell-1)}
        =\frac{h_\ell(z)}{z-1}
\end{equation}
for $z\in U$ with $\Re(z)>1$, and hence by continuation on $U\setminus\{1\}$.
\end{lemma}

\begin{proof}
Let $\mathcal X_\ell$ denote the full group of Dirichlet characters modulo the prime $\ell$.  It has $\ell-1$ elements.  For $\Re(z)>1$ define
\[
        P_\ell(z)=\prod_{\chi\in\mathcal X_\ell}L(z,\chi).
\]
We compare its Euler factors with those of $Z_\ell(z)$.

If $p\neq\ell$ and $p\equiv1\pmod\ell$, then the class of $p$ is the identity in $(\Z/\ell\Z)^\times$, so $\chi(p)=1$ for every $\chi\in\mathcal X_\ell$.  The $p$-Euler factor of $P_\ell$ is therefore
\[
        \prod_{\chi\in\mathcal X_\ell}(1-\chi(p)p^{-z})^{-1}
        =(1-p^{-z})^{-(\ell-1)},
\]
which is exactly the $p$-Euler factor occurring in $Z_\ell$.

Suppose instead that $p\neq\ell$ and $p\not\equiv1\pmod\ell$.  There is no $p$-factor in $Z_\ell$, while the reciprocal of the $p$-factor of $P_\ell$ is
\[
        R_p(z):=\prod_{\chi\in\mathcal X_\ell}(1-\chi(p)p^{-z}).
\]
Expanding this finite product, the coefficient of $p^{-z}$ is
\[
        -\sum_{\chi\in\mathcal X_\ell}\chi(p).
\]
Character orthogonality on the finite group $(\Z/\ell\Z)^\times$ gives
\[
        \sum_{\chi\in\mathcal X_\ell}\chi(p)=0
\]
because the residue class of $p$ is not the identity.  Consequently
\[
        R_p(z)=1+O_\ell(p^{-2\Re(z)}),
\]
locally uniformly in $z$.  At the prime $p=\ell$, every Dirichlet $L$-function modulo $\ell$ omits the Euler factor, and $Z_\ell$ also omits $\ell$, so there is no additional factor.

It follows that the quotient
\[
        R_\ell(z):=\frac{Z_\ell(z)}{P_\ell(z)}
\]
can be represented, initially for $\Re(z)>1$, by an Euler product whose local factors are $1+O_\ell(p^{-2\Re(z)})$.  Hence this product converges absolutely and locally uniformly in every closed half-plane
\[
        \Re(z)\geq\frac12+\delta
\]
with $\delta>0$.  In particular, $R_\ell(z)$ extends to a holomorphic function in a neighbourhood of $z=1$.  Shrinking the neighbourhood if necessary, it is nonzero there: the absolutely convergent Euler product is nonzero at $1$, and holomorphic functions which are nonzero at a point remain nonzero on a sufficiently small neighbourhood.

 The principal character $\chi_0$ modulo $\ell$ satisfies
\[
        L(z,\chi_0)=\zeta(z)(1-\ell^{-z}),
\]
so it has a simple pole at $z=1$.  Every nonprincipal character modulo the prime $\ell$ is primitive, and Dirichlet's nonvanishing theorem gives
\[
        L(1,\chi)\neq0.
\]
Thus
\[
        P_\ell(z)
        =\frac{u_\ell(z)}{z-1}
\]
in a neighbourhood of $1$, where $u_\ell$ is holomorphic and nonzero.  Multiplying by the holomorphic nonzero factor $R_\ell(z)$ yields
\[
        Z_\ell(z)=\frac{h_\ell(z)}{z-1}\quad \text{and}\quad h_\ell(z)=R_\ell(z)u_\ell(z),
\]
\hypersetup{linkcolor=black,citecolor=black,urlcolor=black}and $h_\ell$ is holomorphic and nonvanishing near $1$. Moreover, $h_\ell(1)>0$: it is the nonzero limit of $(z-1)Z_\ell(z)>0$ as real $z$ decreases to $1$. This proves \eqref{eq:zell-factorization}.

The factorization also supplies the continuation required for the Tauberian theorem. Every nonprincipal Dirichlet $L$-function in $P_\ell$ is entire, and its principal factor has only the pole at $1$. The product $R_\ell$ is holomorphic for $\Re z>1/2$. Consequently $Z_\ell=R_\ell P_\ell$ is meromorphic in that half-plane, and its only pole there is the simple pole at $1$. In particular, it is holomorphic at every other point of the line $\Re z=1$.
\end{proof}

\begin{lemma}\label{lem:FQ-factorization}
Put
\[
        H_Q(\mathbf z)
        :=\mathcal F_Q(\mathbf z)
        \prod_{j=1}^sZ_\ell(z_j)^{-1}.
\]
There is a polydisc $\mathcal U$ centred at $(1,\ldots,1)$ such that $H_Q$ is holomorphic and nonvanishing on $\mathcal U$.  Consequently
\begin{equation}\label{eq:FQ-polar-factorization}
        \mathcal F_Q(\mathbf z)
        =\frac{C_Q(\mathbf z)}{(z_1-1)\cdots(z_s-1)},
\end{equation}
where $C_Q$ is holomorphic near $(1,\ldots,1)$ and
\begin{equation}\label{eq:cells-def}
        c_{\ell,s}(Q):=C_Q(1,\ldots,1)>0.
\end{equation}
Moreover,
\begin{align}\label{eq:explicit-cells}
 c_{\ell,s}(Q)
 &=h_\ell(1)^s
 \prod_{\substack{p\equiv1\;({\rm mod}\,\ell)\\p\nmid Q}}
 \left(1+\frac{s(\ell-1)}p\right)
 \left(1-\frac1p\right)^{s(\ell-1)}\notag\\
 &\quad\times
 \prod_{\substack{p\equiv1\;({\rm mod}\,\ell)\\p\mid Q}}
 \left(1-\frac1p\right)^{s(\ell-1)}.
\end{align}
In particular
\begin{equation}\label{eq:cQ-uniform-bound}
        0<c_{\ell,s}(Q)\leq c_{\ell,s}(1)
\end{equation}
for every positive integer $Q$.
\end{lemma}

\begin{proof}
Consider first a prime $p\equiv1\pmod\ell$ with $p\nmid Q$.  The $p$-Euler factor of $H_Q$ is
\begin{equation}\label{eq:HQ-local-factor}
        \left(1+(\ell-1)\sum_{j=1}^sp^{-z_j}\right)
        \prod_{j=1}^s(1-p^{-z_j})^{\ell-1}.
\end{equation}
Write $X_j=p^{-z_j}$.  The first factor in \eqref{eq:HQ-local-factor} has linear part
\[
        1+(\ell-1)(X_1+\cdots+X_s),
\]
while
\[
        \prod_{j=1}^s(1-X_j)^{\ell-1}
        =1-(\ell-1)(X_1+\cdots+X_s)+O_{\ell,s}\!\left(\sum_{i,j}|X_iX_j|\right).
\]
Thus the terms of total degree one cancel, and the local factor is
\[
        1+O_{\ell,s}\!\left(\sum_{i,j}p^{-\Re(z_i)-\Re(z_j)}\right).
\]
Choose $\delta>0$ small.  If every $\Re(z_j)\geq1/2+\delta$, the error is $O_{\ell,s}(p^{-1-2\delta})$.  The corresponding Euler product therefore converges absolutely and locally uniformly in such a region.  This proves holomorphy of the infinite part of $H_Q$ in a neighbourhood of $(1,\ldots,1)$.

If $p\mid Q$ and $p\equiv1\pmod\ell$, then $p$ is excluded from the Euler product defining $\mathcal F_Q$. The only contribution to $H_Q$ is therefore
\[
        \prod_{j=1}^s(1-p^{-z_j})^{\ell-1}.
\]
Only finitely many primes divide $Q$, so these factors cause no convergence problem.  At the point $(1,\ldots,1)$ every local factor is positive.  Absolute convergence shows that the infinite product has a positive value there.  Shrinking the polydisc if necessary, the resulting function is nonvanishing throughout the polydisc.

Lemma~\ref{lem:zell-factorization} gives
\[
        Z_\ell(z_j)=\frac{h_\ell(z_j)}{z_j-1}.
\]
Therefore
\[
        \mathcal F_Q(\mathbf z)
        =H_Q(\mathbf z)
        \prod_{j=1}^s\frac{h_\ell(z_j)}{z_j-1},
\]
which is \eqref{eq:FQ-polar-factorization} with
\[
        C_Q(\mathbf z)=H_Q(\mathbf z)\prod_{j=1}^sh_\ell(z_j).
\]
Since $H_Q(1,\ldots,1)>0$ and $h_\ell(1)>0$, we have $c_{\ell,s}(Q)>0$.

The Euler products give a continuation to a larger domain. For fixed $Q$, $H_Q$ is holomorphic in the product of half-planes $\Re z_j>1/2$. Hence $\mathcal F_Q=H_Q\prod_j Z_\ell(z_j)$ is meromorphic in this product domain, with possible poles only on the coordinate hyperplanes $z_j=1$. This assertion remains valid after taking any fixed collection of partial derivatives. The continuation at points other than $(1,\ldots,1)$ will ensure that the one-variable series obtained on the diagonal satisfy the boundary hypotheses of the Tauberian theorem.

Evaluating the Euler factors at $z_1=\cdots=z_s=1$ gives \eqref{eq:explicit-cells}.  For $p\nmid Q$ the corresponding factor is
\[
        \left(1+\frac{s(\ell-1)}p\right)
        \left(1-\frac1p\right)^{s(\ell-1)}
        =1+O_{\ell,s}(p^{-2}),
\]
so the product converges.  If a prime $p\equiv1\pmod\ell$ is moved from the first product to the second because $p\mid Q$, the constant is multiplied by
\[
        \left(1+\frac{s(\ell-1)}p\right)^{-1}<1.
\]
Primes dividing $Q$ which are not congruent to $1$ modulo $\ell$ do not affect the product at all.  Hence \eqref{eq:cQ-uniform-bound} follows.
\end{proof}

We shall use several consequences of this factorization. We first specify the one-variable Tauberian input. Let $D(z)=\sum_{m\geq1}a_m m^{-z}$ have nonnegative coefficients and converge for $\Re z>1$. Suppose that it extends holomorphically across every point of $\Re z=1$ other than $1$, and that near $1$ it has the form $D(z)=H(z)(z-1)^{-k}$, where $k\geq1$ is an integer, $H$ is holomorphic, and $H(1)>0$. The Ikehara--Delange theorem then gives \[\sum_{m\leq Y}a_m\sim H(1)Y(\log Y)^{k-1}/\Gamma(k);\] see \cite[Th\'eor\`eme~III]{Delange1954}, and \cite{Tenenbaum2015} for the broader analytic framework. To pass from the Laplace-transform formulation to this statement, apply the theorem to $D(z)/z$, the Laplace transform of the nondecreasing function $t\mapsto\sum_{m\leq e^t}a_m$. The leading coefficient at $z=1$ is unchanged. The continuation along the rest of the boundary is part of the hypothesis and will be verified in each application.

\begin{lemma}\label{lem:box-counting}
Uniformly in $Q$ and $Y\geq1$,
\begin{equation}\label{eq:product-upper-bound}
        \sum_{\substack{q_1\cdots q_s\leq Y\\q_i\text{ as above}}}
        w(\mathbf q)
        \ll_{\ell,s}Y(\log 2Y)^{s-1}.
\end{equation}
\end{lemma}

\begin{proof}
For the upper bound, we enlarge the summation range so that the coordinates can be counted independently. Discard the pairwise coprimality condition and the requirement that $q_1\cdots q_s$ be coprime to $Q$.  Define
\[
        a_\ell(q)=\mu^2(q)(\ell-1)^{\omega(q)}
        1_{\{p\mid q\Rightarrow p\equiv1\;({\rm mod}\,\ell)\}}.
\]
Then the enlarged weighted sum over $q_1\cdots q_s\leq Y$ is
\[
        \sum_{m\leq Y}a_\ell^{*s}(m),
\]
where $a_\ell^{*s}$ denotes the $s$-fold Dirichlet convolution.  The Dirichlet series of $a_\ell$ is
\[
        A_\ell(z)=
        \prod_{p\equiv1\;({\rm mod}\,\ell)}
        \left(1+(\ell-1)p^{-z}\right).
\]
Multiplying and dividing by $Z_\ell(z)$, one obtains
\[
        A_\ell(z)=Z_\ell(z)U_\ell(z),
\]
where
\[
        U_\ell(z)=
        \prod_{p\equiv1\;({\rm mod}\,\ell)}
        \left(1+(\ell-1)p^{-z}\right)
        (1-p^{-z})^{\ell-1}
\]
converges absolutely for $\Re z>1/2$, because again the terms of degree one cancel. Lemma~\ref{lem:zell-factorization} therefore shows that $A_\ell(z)$ is meromorphic in this half-plane, with only a simple pole at $1$ and a holomorphic nonzero factor there. Hence $A_\ell(z)^s$, the Dirichlet series of $a_\ell^{*s}$, has a pole of order $s$ at $1$ and is holomorphic at every other point of $\Re z=1$. Its coefficients are nonnegative. The full boundary hypothesis has therefore been checked, and the Ikehara--Delange theorem stated above, or the Selberg--Delange method as in \cite{Tenenbaum2015}, gives
\[
        \sum_{m\leq Y}a_\ell^{*s}(m)
        \ll_{\ell,s}Y(\log 2Y)^{s-1}.
\]
\hypersetup{linkcolor=black,citecolor=black,urlcolor=black}The original sum is bounded by this enlarged sum, and the estimate is independent of $Q$. For $1\leq Y\leq2$, the same bound follows by monotonicity from the bound at $Y=2$. This proves \eqref{eq:product-upper-bound}.
\end{proof}

The next consequence is the distribution of the logarithms under a product constraint. We give the proof because this is where the simplex measure appears.

\begin{lemma}\label{lem:minimal-counting}
As $Y\to\infty$,
\begin{equation}\label{eq:product-count}
        \sum_{\substack{q_1\cdots q_s\leq Y\\q_i\text{ as above}}}
        w(\mathbf q)
        \sim
        \frac{c_{\ell,s}(Q)}{(s-1)!}
        Y(\log Y)^{s-1}.
\end{equation}
Let
\[
        \Delta_{s-1}
        =\{(t_1,\ldots,t_s)\in\R_{\geq0}^s:t_1+\cdots+t_s=1\},
\]
with $d\mathbf t=dt_1\cdots dt_{s-1}$ and $t_s=1-t_1-\cdots-t_{s-1}$. When $s=1$, this means unit mass at the point $\Delta_0=\{1\}$. If $F$ is continuous on the compact simplex $\{\mathbf t\geq0:\sum_i t_i\leq1\}$, then
\begin{align}\label{eq:simplex-count}
&\sum_{\substack{q_1\cdots q_s\leq Y\\q_i\text{ as above}}}
 w(\mathbf q)
 F\left(\frac{\log q_1}{\log Y},\ldots,
 \frac{\log q_s}{\log Y}\right)\notag\\
&\qquad=
 c_{\ell,s}(Q)Y(\log Y)^{s-1}
 \int_{\Delta_{s-1}}F(\mathbf t)\,d\mathbf t
 +o_{F,Q}\!\left(Y(\log Y)^{s-1}\right).
\end{align}
\end{lemma}

\begin{proof}
We compute the mixed moments of the normalized logarithms and compare them with the moments of normalized Lebesgue measure on $\Delta_{s-1}$. The moment with all exponents zero gives the total count.

For a multi-index $\mathbf k=(k_1,\ldots,k_s)\in\Z_{\geq0}^s$, define $|\mathbf k|=k_1+\cdots+k_s$ and
set
\begin{equation}\label{eq:Mk-def}
        M_{\mathbf k}(Y;Q)
        =\sum_{q_1\cdots q_s\leq Y}
        w(\mathbf q)\prod_{i=1}^s(\log q_i)^{k_i},
\end{equation}
where the sum runs over ordered tuples $(q_1,\ldots,q_s)$
of pairwise coprime positive squarefree integers such that
$(q_1\cdots q_s,Q)=1$ and every prime divisor of
$q_1\cdots q_s$ is congruent to $1$ modulo $\ell$.
For $\Re(z)>1$, absolute convergence allows us to differentiate the multiple Dirichlet series term by term.  After differentiating in the $i$-th variable $k_i$ times and then restricting to the diagonal $z_1=\cdots=z_s=z$, we obtain
\begin{align}\label{eq:Dk-def}
        D_{\mathbf k,Q}(z)
        &:=\sum_{\mathbf q}
        \frac{w(\mathbf q)\prod_i(\log q_i)^{k_i}}
        {(q_1\cdots q_s)^z}\notag\\
        &=(-1)^{|\mathbf k|}
        \left.
        \frac{\partial^{|\mathbf k|}}
        {\partial z_1^{k_1}\cdots\partial z_s^{k_s}}
        \mathcal F_Q(\mathbf z)
        \right|_{z_1=\cdots=z_s=z}.
\end{align}

Insert the factorization from Lemma~\ref{lem:FQ-factorization}:
\[
        \mathcal F_Q(\mathbf z)
        =C_Q(\mathbf z)\prod_{i=1}^s(z_i-1)^{-1}.
\]
The highest possible pole order in \eqref{eq:Dk-def} occurs when, for every $i$, all $k_i$ derivatives in the variable $z_i$ fall on the factor $(z_i-1)^{-1}$.  Since
\[
        \frac{d^{k_i}}{dz_i^{k_i}}(z_i-1)^{-1}
        =(-1)^{k_i}k_i!(z_i-1)^{-k_i-1},
\]
the signs cancel the prefactor $(-1)^{|\mathbf k|}$ in \eqref{eq:Dk-def}.  The resulting leading term is
\[
        C_Q(1,\ldots,1)
        \left(\prod_{i=1}^sk_i!\right)
        (z-1)^{-s-|\mathbf k|}.
\]
If at least one derivative falls on $C_Q$, then at least one polar factor receives fewer derivatives, and the total pole order drops by at least one.  Because $C_Q$ and all of its derivatives are holomorphic near $(1,\ldots,1)$, we therefore have
\begin{equation}\label{eq:moment-pole}
        D_{\mathbf k,Q}(z)
        =c_{\ell,s}(Q)
        \left(\prod_{i=1}^sk_i!\right)
        (z-1)^{-s-|\mathbf k|}
        +O_{\mathbf k,Q}\left((z-1)^{-s-|\mathbf k|+1}\right).
\end{equation}

The coefficients of $D_{\mathbf k,Q}$ are nonnegative, and its series converges absolutely for $\Re z>1$. To apply the Ikehara--Delange theorem, we also check continuation along the rest of the boundary. By Lemma~\ref{lem:FQ-factorization}, $\mathcal F_Q$ is meromorphic in the product domain $\Re z_i>1/2$, with possible poles only on $z_i=1$. Partial differentiation preserves this property: the differentiated factors remain holomorphic away from these coordinate hyperplanes. Restricting the resulting expressions to the diagonal gives a meromorphic continuation of $D_{\mathbf k,Q}$ to $\Re z>1/2$ whose only possible pole is at $1$. Thus it is holomorphic at every other point of the full boundary line $\Re z=1$.

Near $1$, equation \eqref{eq:moment-pole} can be written as $D_{\mathbf k,Q}(z)=H_{\mathbf k,Q}(z)(z-1)^{-\rho}$ with $\rho=s+|\mathbf k|$, where $H_{\mathbf k,Q}$ is holomorphic and $H_{\mathbf k,Q}(1)=c_{\ell,s}(Q)\prod_i k_i!>0$. All the hypotheses of the Tauberian statement preceding Lemma~\ref{lem:box-counting} are now satisfied. Applying it with this pole order and leading coefficient gives
\begin{equation}\label{eq:moment-asymptotic}
        M_{\mathbf k}(Y;Q)
        \sim c_{\ell,s}(Q)
        \frac{\prod_i k_i!}{(s+|\mathbf k|-1)!}
        Y(\log Y)^{s+|\mathbf k|-1}.
\end{equation}
For $\mathbf k=\mathbf0$, this is precisely \eqref{eq:product-count}.

We now turn the moment asymptotics into weak convergence.  Put $T=\log Y$ and let $S_Q(Y)$ denote the total sum in \eqref{eq:product-count}.  Define a probability measure $\mu_Y$ on
\[
        \Sigma_s:=\{\mathbf x\in\R_{\geq0}^s:x_1+\cdots+x_s\leq1\}
\]
by
\[
        \mu_Y
        =\frac1{S_Q(Y)}
        \sum_{\substack{q_1\cdots q_s\leq Y\\q_i\text{ as above}}}
        w(\mathbf q)
        \delta_{(\log q_1/T,\ldots,\log q_s/T)}.
\]
The support lies in $\Sigma_s$ because the product constraint is equivalent to
\[
        \frac{\log q_1}{T}+\cdots+\frac{\log q_s}{T}\leq1.
\]
For the monomial $\mathbf x^{\mathbf k}=x_1^{k_1}\cdots x_s^{k_s}$, equations \eqref{eq:product-count} and \eqref{eq:moment-asymptotic} give
\begin{align*}
        \int_{\Sigma_s}\mathbf x^{\mathbf k}\,d\mu_Y(\mathbf x)
        &=\frac{M_{\mathbf k}(Y;Q)}{T^{|\mathbf k|}S_Q(Y)}\\
        &\longrightarrow
        \frac{(s-1)!\prod_i k_i!}
        {(s+|\mathbf k|-1)!}.
\end{align*}

Let $\mu_\Delta$ be normalized Lebesgue measure on the outer face
\[
        \Delta_{s-1}=\{\mathbf t\in\R_{\geq0}^s:t_1+\cdots+t_s=1\}.
\]
With the coordinates $t_1,\ldots,t_{s-1}$ and $t_s=1-\sum_{i<s}t_i$, ordinary Lebesgue measure on this simplex has total mass $1/(s-1)!$.  Thus
\[
        d\mu_\Delta=(s-1)!\,d\mathbf t.
\]
The Dirichlet beta integral yields
\[
        \int_{\Delta_{s-1}}\prod_{i=1}^st_i^{k_i}\,d\mathbf t
        =\frac{\prod_i k_i!}{(s+|\mathbf k|-1)!},
\]
and hence
\[
        \int\mathbf t^{\mathbf k}\,d\mu_\Delta(\mathbf t)
        =\frac{(s-1)!\prod_i k_i!}{(s+|\mathbf k|-1)!}.
\]
The limiting moments of $\mu_Y$ are therefore exactly the moments of $\mu_\Delta$.

Since $\Sigma_s$ is compact, polynomials in the coordinate functions form a subalgebra of $C(\Sigma_s)$ which contains the constants and separates points.  By the Stone--Weierstrass theorem they are uniformly dense in $C(\Sigma_s)$.  Convergence of all polynomial moments therefore implies
\[
        \mu_Y\Longrightarrow\mu_\Delta
\]
weakly.

If $F$ is continuous on $\Sigma_s$, weak convergence gives
\[
        \frac1{S_Q(Y)}
        \sum_{q_1\cdots q_s\leq Y}w(\mathbf q)
        F\left(\frac{\log q_1}{T},\ldots,\frac{\log q_s}{T}\right)
        \longrightarrow
        (s-1)!\int_{\Delta_{s-1}}F(\mathbf t)\,d\mathbf t.
\]
Multiplying by the asymptotic
\[
        S_Q(Y)\sim\frac{c_{\ell,s}(Q)}{(s-1)!}Y(\log Y)^{s-1}
\]
gives \eqref{eq:simplex-count} with an additive error. This formulation also applies when the limiting integral is zero, or when $F$ takes both positive and negative values.
\end{proof}

Return to the fixed group $G$. Let $B_G$ be the submatrix of $A_G$ whose columns are indexed by the minimal inertia subgroups $\cM(G)$. Define
\begin{equation}\label{eq:shape-simplex}
        \mathscr P_G
        :=\operatorname{conv}\{v_C:C\in\cM(G)\}\subset E_G.
\end{equation}
The following lemma establishes the affine independence of its vertices.

\begin{lemma}\label{lem:simplex-independent}
The points $\{v_C:C\in\cM(G)\}$ are affinely independent. Hence $\mathscr P_G$ has dimension $s_G-1$.
\end{lemma}

\begin{proof}
Affine independence means that the only relation
\[
        \sum_{C\in\cM(G)}r_Cv_C=0
\]
whose coefficients satisfy
\[
        \sum_{C\in\cM(G)}r_C=0
\]
is the trivial relation. Suppose such a relation is given, and write $\mathbf r=(r_C)_{C\in\cM(G)}$.  By definition $v_C=\pi(A_C)$, and the minimal columns of $A_G$ form the matrix $B_G$.  Hence
\[
        \pi(B_G\mathbf r)=0.
\]
The kernel of $\pi:\R^{\Omega(G)}\to E_G$ is exactly the line $\R\mathbf1$.  There is therefore a scalar $c\in\R$ such that
\begin{equation}\label{eq:BGr-c1}
        B_G\mathbf r=c\mathbf1.
\end{equation}
We consider separately the cases $G[\ell]\neq G$ and $G[\ell]=G$.

Assume first that $G$ is not elementary abelian of exponent $\ell$.  Then $G[\ell]\neq G$.  Indeed, if every element of $G$ were killed by $\ell$, the finite abelian group $G$ would be a vector space over $\F_\ell$ and hence isomorphic to $(C_\ell)^r$.  Thus the quotient $G/G[\ell]$ is a nontrivial finite abelian group.  It has a nontrivial complex character; inflating such a character to $G$ gives a character $\chi$ which is trivial on $G[\ell]$ and hence on every subgroup $C\in\cM(G)$.  The rational Galois orbit of $\chi$ therefore determines a row of $B_G$ all of whose entries are zero.  Reading this row in \eqref{eq:BGr-c1} gives $c=0$.

We are left with $B_G\mathbf r=0$.  The columns of $B_G$ are a subset of the columns of the invertible square matrix $A_G$.  Every subset of the columns of an invertible matrix is linearly independent, so $\mathbf r=0$.

Now suppose that $G=(C_\ell)^r$.  In this case every nontrivial cyclic subgroup has order $\ell$, so $B_G=A_G$.  A nontrivial character has kernel a hyperplane $H\subset\F_\ell^r$.  There are
\[
        \frac{\ell^r-1}{\ell-1}
\]
lines in $\F_\ell^r$ and
\[
        \frac{\ell^{r-1}-1}{\ell-1}
\]
lines in $H$.  Therefore exactly
\[
        \frac{\ell^r-\ell^{r-1}}{\ell-1}=\ell^{r-1}
\]
lines are not contained in $H$.  By the definition of $A_G$, every row therefore contains $\ell^{r-1}$ entries equal to $1$.  Hence
\[
        A_G\mathbf1=\ell^{r-1}\mathbf1.
\]
Since $A_G$ is invertible,
\[
        A_G^{-1}\mathbf1=\ell^{1-r}\mathbf1.
\]
Equation \eqref{eq:BGr-c1} now becomes $A_G\mathbf r=c\mathbf1$, and applying $A_G^{-1}$ gives
\[      \mathbf r=c\ell^{1-r}\mathbf1.
\]
But the coefficients of an affine relation also satisfy $\sum_Cr_C=0$.  A constant vector has zero coordinate sum only when it is zero.  Therefore $c=0$ and $\mathbf r=0$.

In both cases the only affine relation is trivial.  Thus the points $v_C$ are affinely independent, and the convex hull of $s_G$ such points has dimension $s_G-1$.
\end{proof}

To state the distribution theorem, let $\mathscr F_G^{\rm mark}(X)$ denote the set of marked tame $G$-extensions $(K,\iota)$ with $|\disc(K)|\leq X$, and let $\mathscr F_G(X)$ denote the corresponding set of unmarked fields. The group $\Aut(G)$ acts freely on $\mathscr F_G^{\rm mark}(X)$ by changing the marking, and the quotient is $\mathscr F_G(X)$.

\begin{theorem}\label{thm:general-shape-distribution}
Let $\ell$ be the smallest prime divisor of $|G|$, and let $\alpha_G$ be as in \eqref{eq:a0-alpha}. There is an explicit constant $C_G^{\rm mark}>0$ such that
\begin{equation}\label{eq:total-G-count-marked}
        |\mathscr F_G^{\rm mark}(X)|
        \sim C_G^{\rm mark}X^{\alpha_G}(\log X)^{s_G-1}.
\end{equation}
Moreover, for every bounded continuous function $F$ on $E_G$,
\begin{align}\label{eq:shape-limit-marked}
&\frac1{|\mathscr F_G^{\rm mark}(X)|}
\sum_{(K,\iota)\in\mathscr F_G^{\rm mark}(X)}
F\left(\frac{a_0}{\log X}\lambda(K,\iota)\right)\notag\\
&\qquad\longrightarrow
(s_G-1)!\int_{\Delta_{s_G-1}}
F\left(\sum_{C\in\cM(G)}t_Cv_C\right)d\mathbf t.
\end{align}
Thus the marked rescaled logarithmic conductor parameters become uniformly distributed on $\mathscr P_G$.

For unmarked fields one has
\begin{equation}\label{eq:total-G-count}
        |\mathscr F_G(X)|
        \sim C_GX^{\alpha_G}(\log X)^{s_G-1},
       \quad \text{where}\quad C_G=\frac{C_G^{\rm mark}}{|\Aut(G)|}.
\end{equation}
For unmarked fields, the rescaled canonical classes $a_0\overline\lambda(K)/\log X$ are distributed according to the pushforward of normalized Lebesgue measure on $\mathscr P_G$ under the quotient map $E_G\to E_G/\Aut(G)$. Equivalently, \eqref{eq:shape-limit-marked} also holds for unmarked fields whenever $F$ is $\Aut(G)$-invariant.
\end{theorem}

\begin{proof}
Write
\[
        \cC(G)=\cM(G)\sqcup\cC_{>}(G),
\]
where $\cC_{>}(G)$ consists of the cyclic subgroups of order strictly larger than $\ell$.  Fix for the moment an admissible tuple in the higher coordinates,
\[
        \mathbf q_{>}=(q_C)_{C\in\cC_{>}(G)}.
\]
We use the notation
\begin{equation}\label{eq:higher-notation}
        Q_{>}=
        \prod_{C\in\cC_{>}(G)}q_C,
        \qquad
        D_{>}=
        \prod_{C\in\cC_{>}(G)}q_C^{a_C}\quad \text{and}\quad
        w_{>}=
        \prod_{C\in\cC_{>}(G)}
        \varphi(d_C)^{\omega(q_C)}.
\end{equation}
Here $Q_{>}$ is the product of the primes in the higher coordinates, $D_{>}$ is their contribution to the discriminant, and $w_{>}$ counts the corresponding choices of local surjections. Pairwise coprimality requires the minimal coordinates to be coprime to $Q_{>}$.

For every $C\in\cM(G)$ one has $a_C=a_0$.  Thus the discriminant formula of Lemma~\ref{lem:tame-discriminant} becomes
\[
        D_{>}
        \left(\prod_{C\in\cM(G)}q_C\right)^{a_0}
        \leq X.
\]
Since $\alpha_G=1/a_0$, this is equivalent to
\begin{equation}\label{eq:Y-higher}
        \prod_{C\in\cM(G)}q_C
        \leq Y_{\mathbf q_{>}}\quad \text{and}\quad
        Y_{\mathbf q_{>}}
        :=\left(\frac{X}{D_{>}}\right)^{\alpha_G}.
\end{equation}
If $D_{>}>X$ there is no contribution.  Otherwise $Y_{\mathbf q_{>}}\geq1$.

The minimal coordinates have a uniform arithmetic description.  Every prime occurring in them is congruent to $1$ modulo $\ell$, every local surjection has multiplicity
\[
        \varphi(\ell)=\ell-1,
\]
and the primes must avoid $Q_{>}$.

The collection of all order-$\ell$ subgroups of $G$ generates $G[\ell]$.  Consequently a necessary condition for the fixed higher tuple to occur in the leading term is
\begin{equation}\label{eq:higher-generation}
        \left\langle
        G[\ell],
        C:q_C>1,\ C\in\cC_{>}(G)
        \right\rangle=G.
\end{equation}
If this condition fails, adjoining all the minimal inertia subgroups to the active higher subgroups still gives a proper subgroup of $G$. Thus no choice of minimal coordinates makes the map in Lemma~\ref{lem:tame-parametrization} surjective.

Assume now that \eqref{eq:higher-generation} holds.  If every minimal coordinate is nontrivial, then every subgroup in $\cM(G)$ occurs as an inertia image.  These subgroups generate $G[\ell]$, and together with the active higher subgroups they generate $G$ by \eqref{eq:higher-generation}.  Thus every such tuple gives a surjective global map.

We must check that allowing some minimal coordinate to equal $1$ does not change the leading term.  Fix $C_0\in\cM(G)$ and impose $q_{C_0}=1$.  Dropping the generation condition gives an upper bound, and the remaining $s_G-1$ minimal coordinates satisfy the same product constraint.  Lemma~\ref{lem:box-counting}, applied with $s_G-1$ coordinates, gives
\[
        O_G\left(
        Y_{\mathbf q_{>}}
        (\log 2Y_{\mathbf q_{>}})^{s_G-2}
        \right)
\]
when $s_G\geq2$.  There are only $s_G$ choices of a missing coordinate, so the same estimate holds for the union of all tuples with at least one missing minimal coordinate.  If $s_G=1$, the only tuple with the unique minimal coordinate absent is $q_C=1$, and its contribution is $O(1)$, which is again lower order than the main term $\asymp Y_{\mathbf q_{>}}$. Therefore, for a fixed higher tuple, imposing the generation condition changes the minimal-coordinate count by $o\bigl(Y_{\mathbf q_{>}}(\log Y_{\mathbf q_{>}})^{s_G-1}\bigr)$.

Lemma~\ref{lem:minimal-counting} now gives, for a fixed higher tuple satisfying \eqref{eq:higher-generation},
\begin{align}\label{eq:fixed-higher-main-detailed}
N_{\mathbf q_{>}}^{\rm mark}(X)
&=
w_{>}
\frac{c_{\ell,s_G}(Q_{>})}{(s_G-1)!}
Y_{\mathbf q_{>}}
(\log Y_{\mathbf q_{>}})^{s_G-1}
\bigl(1+o_{\mathbf q_{>}}(1)\bigr).
\end{align}
For a fixed higher tuple,
\[
        \log Y_{\mathbf q_{>}}
        =\alpha_G\log X-\alpha_G\log D_{>}
        =\alpha_G\log X+O_{\mathbf q_{>}}(1),
\]
and
\[
        Y_{\mathbf q_{>}}
        =X^{\alpha_G}D_{>}^{-\alpha_G}.
\]
Hence \eqref{eq:fixed-higher-main-detailed} is equivalent to
\begin{align}\label{eq:fixed-higher-main}
N_{\mathbf q_{>}}^{\rm mark}(X)
&=
\frac{w_{>}c_{\ell,s_G}(Q_{>})}
{(s_G-1)!D_{>}^{\alpha_G}}
X^{\alpha_G}(\alpha_G\log X)^{s_G-1}
\bigl(1+o_{\mathbf q_{>}}(1)\bigr).
\end{align}

The fixed-tuple asymptotic cannot be summed over infinitely many tuples without a uniform tail estimate, so we truncate first.  Fix $M\geq1$ and sum only over admissible higher tuples satisfying \eqref{eq:higher-generation} and
\[
        D_{>}\leq M.
\]
There are only finitely many such tuples: every prime dividing $Q_{>}$ divides the integer $D_{>}$, so only primes at most $M$ may occur.  The error terms in \eqref{eq:fixed-higher-main} can therefore be summed term by term over the truncated collection.

We next prove that the series defining the leading constant converges absolutely. If $C\in\cC_{>}(G)$, then $a_C>a_0$, and therefore
\begin{equation}\label{eq:strict-higher-exponent}
        \alpha_Ga_C=\frac{a_C}{a_0}>1.
\end{equation}
Drop all coprimality and generation conditions.  The resulting majorant factors over the finitely many higher cyclic subgroups:
\begin{align*}
        \sum_{\mathbf q_{>}}
        \frac{w_{>}}{D_{>}^{\alpha_G}}
        &\leq
        \prod_{C\in\cC_{>}(G)}
        \prod_{p\equiv1\;({\rm mod}\,d_C)}
        \left(1+
        \frac{\varphi(d_C)}{p^{\alpha_Ga_C}}
        \right).
\end{align*}
Each Euler product converges by \eqref{eq:strict-higher-exponent}.  Hence
\begin{equation}\label{eq:higher-series-converges}
        \sum_{\mathbf q_{>}}
        \frac{w_{>}}{D_{>}^{\alpha_G}}<\infty.
\end{equation}
Lemma~\ref{lem:FQ-factorization} also gives the uniform bound
\[
        c_{\ell,s_G}(Q_{>})\leq c_{\ell,s_G}(1).
\]
Thus the same series remains convergent after inserting the factor $c_{\ell,s_G}(Q_{>})$.

We now control the arithmetic tail uniformly in $X$.  For a fixed higher tuple, Lemma~\ref{lem:box-counting} gives
\begin{align*}
N_{\mathbf q_{>}}^{\rm mark}(X)
&\ll_G
w_{>}
Y_{\mathbf q_{>}}
(\log 2Y_{\mathbf q_{>}})^{s_G-1}\\
&\ll_G
X^{\alpha_G}(\log 2X)^{s_G-1}
\frac{w_{>}}{D_{>}^{\alpha_G}}.
\end{align*}
Consequently the contribution of all higher tuples with $D_{>}>M$ is at most
\begin{equation}\label{eq:higher-tail-count}
        \ll_G
        X^{\alpha_G}(\log 2X)^{s_G-1}
        \sum_{D_{>}>M}
        \frac{w_{>}}{D_{>}^{\alpha_G}}.
\end{equation}
After division by $X^{\alpha_G}(\log X)^{s_G-1}$, the right-hand side tends to zero as $M\to\infty$, uniformly for large $X$, by \eqref{eq:higher-series-converges}.  We may therefore first let $X\to\infty$ for the finite truncated collection and then let $M\to\infty$.

This proves \eqref{eq:total-G-count-marked} with
\begin{equation}\label{eq:CG}
        C_G^{\rm mark}=
        \frac{\alpha_G^{s_G-1}}
        {(s_G-1)!}
        \sum_{\mathbf q_{>}}^{*}
        \frac{w_{>}c_{\ell,s_G}(Q_{>})}
        {D_{>}^{\alpha_G}},
\end{equation}
where the star means that the higher tuple is admissible and satisfies \eqref{eq:higher-generation}.  The series is absolutely convergent by the preceding argument.

The constant is positive.  If $G=G[\ell]$, the empty higher tuple satisfies the generation condition and already gives a positive term.  Otherwise choose finitely many elements of $G$ whose nontrivial images generate $G/G[\ell]$ and let $C$ be the cyclic subgroup generated by each chosen element.  Every such $C$ has order strictly larger than $\ell$.  By Dirichlet's theorem on primes in arithmetic progressions, for each of these cyclic subgroups we may choose a distinct prime
\[
        p\equiv1\pmod{|C|}.
\]
Putting that prime into the corresponding coordinate produces an admissible higher tuple satisfying \eqref{eq:higher-generation}, and its summand in \eqref{eq:CG} is positive.

For $C\in\cM(G)$ write
\[
        u_C=\log q_C.
\]
For a fixed higher tuple, the decomposition \eqref{eq:log-shape} becomes
\begin{equation}\label{eq:lambda-split}
        \lambda(K,\iota)
        =\sum_{C\in\cM(G)}u_Cv_C
        +\sum_{C\in\cC_{>}(G)}(\log q_C)v_C.
\end{equation}
Set
\[
        T_{>}:=\log Y_{\mathbf q_{>}}
        =\alpha_G\log X-\alpha_G\log D_{>}.
\]
Under the product constraint, $0\leq u_C\leq T_{>}$.  Since $a_0=1/\alpha_G$,
\[
        \frac{a_0}{\log X}u_C
        =\frac{u_C}{\alpha_G\log X}.
\]
For fixed $\mathbf q_{>}$,
\begin{align*}
\left|
\frac{u_C}{\alpha_G\log X}-\frac{u_C}{T_{>}}
\right|
&\leq
T_{>}
\left|
\frac1{\alpha_G\log X}-\frac1{T_{>}}
\right|\\
&=
\frac{\alpha_G\log D_{>}}{\alpha_G\log X}
=O_{\mathbf q_{>}}\left(\frac1{\log X}\right).
\end{align*}
The higher-coordinate term in \eqref{eq:lambda-split}, after multiplication by $a_0/\log X$, is also $O_{\mathbf q_{>}}(1/\log X)$.  Hence, uniformly over all minimal tuples satisfying the product constraint,
\begin{equation}\label{eq:shape-uniform-approx}
        \frac{a_0}{\log X}\lambda(K,\iota)
        =\sum_{C\in\cM(G)}
        \frac{u_C}{T_{>}}v_C+o_{\mathbf q_{>}}(1).
\end{equation}

Let $F$ be bounded and continuous on $E_G$.  The vectors $\left(\frac{u_C}{T_{>}}\right)_{C\in\cM(G)}$
lie in the compact simplex $\sum_C t_C\leq1$.  The main terms on the right of \eqref{eq:shape-uniform-approx} therefore lie in a fixed compact subset of $E_G$.  Uniform continuity of $F$ on a slightly larger compact neighbourhood, together with \eqref{eq:shape-uniform-approx}, shows that replacing the true argument of $F$ by
\[
        \sum_C\frac{u_C}{T_{>}}v_C
\]
changes the weighted sum by
\[
        o_{\mathbf q_{>}}\left(
        Y_{\mathbf q_{>}}T_{>}^{s_G-1}
        \right).
\]

\hypersetup{linkcolor=black,citecolor=black,urlcolor=black}For a fixed higher tuple, the tuples excluded by the generation condition contribute at most $\|F\|_\infty$ times the error estimated in Step~2, so their weighted contribution is still negligible. We may therefore apply the second assertion of Lemma~\ref{lem:minimal-counting} to the continuous function
\[
        (t_C)_C\longmapsto
        F\left(\sum_Ct_Cv_C\right).
\]
The contribution of the fixed higher tuple equals its arithmetic main factor in \eqref{eq:fixed-higher-main}, multiplied by the following expression, with an additive error of size $o_{F,\mathbf q_{>}}(X^{\alpha_G}(\log X)^{s_G-1})$:
\begin{equation}\label{eq:shape-simplex-integral}
        (s_G-1)!
        \int_{\Delta_{s_G-1}}
        F\left(\sum_{C\in\cM(G)}t_Cv_C\right)d\mathbf t.
\end{equation}
This expression gives the contribution relative to the unweighted main term. More explicitly, before normalization the factor supplied by Lemma~\ref{lem:minimal-counting} is
\[
        \int_{\Delta_{s_G-1}}
        F\left(\sum_Ct_Cv_C\right)d\mathbf t
\]
in place of $1/(s_G-1)!$.

As in the counting argument, first sum over $D_{>}\leq M$.  For the tail, the absolute value of the shape-weighted contribution is at most $\|F\|_\infty$ times the unweighted contribution.  The uniform estimate \eqref{eq:higher-tail-count} therefore applies without change.  We may let $X\to\infty$ and then $M\to\infty$.  Dividing by the total marked count causes the absolutely convergent arithmetic series in \eqref{eq:CG} to cancel, and yields exactly \eqref{eq:shape-limit-marked}.

By Lemma~\ref{lem:simplex-independent}, the affine map
\[
        \Delta_{s_G-1}\longrightarrow\mathscr P_G 
\]
defined by $        (t_C)_C\longmapsto\sum_Ct_Cv_C$ is a bijection.  An affine bijection sends Lebesgue measure to a constant multiple of Lebesgue measure on the affine span.  Since $(s_G-1)!d\mathbf t$ has total mass one, the limiting measure is precisely normalized uniform Lebesgue measure on $\mathscr P_G$.

The group $\Aut(G)$ acts freely on the set of markings of each field, and every unmarked field with Galois group isomorphic to $G$ has exactly $|\Aut(G)|$ markings.  Therefore
\[
        |\mathscr F_G^{\rm mark}(X)|
        =|\Aut(G)|\,|\mathscr F_G(X)|,
\]
which gives \eqref{eq:total-G-count} and
\[
        C_G=\frac{C_G^{\rm mark}}{|\Aut(G)|}.
\]
If $F$ is $\Aut(G)$-invariant, its value is unchanged when the marking is replaced by $\alpha\circ\iota$.  The marked numerator is then exactly $|\Aut(G)|$ times the corresponding unmarked numerator.  Dividing numerator and denominator by this common factor gives the same limit for unmarked fields.  Thus the rescaled classes $a_0\overline\lambda(K)/\log X$ in $E_G/\Aut(G)$ have the pushforward of the simplex measure as their limiting distribution.
\end{proof}

The simplex in the distribution theorem arises from the product
constraint on the inertia coordinates with the smallest discriminant
exponent. Each subgroup $C\in\cM(G)$ contributes the factor
$q_C^{a_0}$ to the discriminant. After fixing the remaining inertia
coordinates, the discriminant bound therefore becomes
$\prod_{C\in\cM(G)}q_C\leq Y$, where
$Y=(X/D_{>})^{1/a_0}$ and $D_{>}$ is the discriminant contribution
from those remaining coordinates. The normalized logarithms
$t_C=\log q_C/\log Y$ are nonnegative and satisfy
$\sum_Ct_C\leq1$.

Under the weighted counting measure, however,
the contribution from tuples with
$\prod_Cq_C\leq Y^{1-\varepsilon}$ is negligible for every fixed
$\varepsilon>0$, so the limiting measure is supported on the face
$\sum_Ct_C=1$. Indeed, let $S_Q(U)$ denote the total weight of the
allowed tuples satisfying $\prod_Cq_C\leq U$. For fixed $Q$,
Lemma~\ref{lem:minimal-counting} gives
\[
        S_Q(U)\sim
        \frac{c_{\ell,s}(Q)}{(s-1)!}
        U(\log U)^{s-1}.
\]
Thus, for every fixed $0<\varepsilon<1$, the proportion of
the total weight coming from tuples with
$\prod_Cq_C\leq Y^{1-\varepsilon}$ satisfies
\[
        \frac{S_Q(Y^{1-\varepsilon})}{S_Q(Y)}
        \sim
        (1-\varepsilon)^{s-1}Y^{-\varepsilon}
        \longrightarrow0.
\]
Since
\[
        \sum_Ct_C=\frac{\log(\prod_Cq_C)}{\log Y},
\]
these are exactly the tuples for which
$\sum_Ct_C\leq1-\varepsilon$. Their contribution therefore
becomes negligible for every fixed $\varepsilon>0$.
On the other hand, the product bound ensures that
$\sum_Ct_C\leq1$ for every tuple. Hence the sum of the
normalized logarithms tends to one in probability.

The mixed-moment calculation in
Lemma~\ref{lem:minimal-counting} proves that this measure is the
uniform probability measure on the standard simplex. The estimates
in Step~3 of the proof of
Theorem~\ref{thm:general-shape-distribution} show that the higher
inertia coordinates contribute negligibly after division by
$\log X$. Step~4 then transfers the preceding limit to the
rescaled logarithmic conductor parameters through the map
$(t_C)_C\mapsto\sum_Ct_Cv_C$. Its image is
$\mathscr P_G=\operatorname{conv}\{v_C:C\in\cM(G)\}$, and
Lemma~\ref{lem:simplex-independent} proves that the vertices are
affinely independent. Thus this map identifies the standard simplex
with $\mathscr P_G$ and carries its uniform probability measure
to the limiting measure in the distribution theorem.

\bibliographystyle{alpha}
\bibliography{references}

\end{document}